\documentclass[11pt]{article}
\usepackage[margin=1in]{geometry}
\usepackage{amsmath,amssymb,amsthm,booktabs}
\usepackage[hidelinks]{hyperref}
\usepackage{natbib}
\usepackage{longtable,array,microtype}
\newtheorem{theorem}{Theorem}[section]
\newtheorem{lemma}[theorem]{Lemma}
\newtheorem{proposition}[theorem]{Proposition}
\newtheorem{corollary}[theorem]{Corollary}
\theoremstyle{remark}

\newcommand{\E}{\mathbb E}
\newcommand{\Pp}{\mathbb P}
\newcommand{\Var}{\operatorname{Var}}
\newcommand{\ind}{\mathbf 1}
\newcommand{\dd}{\mathrm d}
\newcommand{\TV}{d_{\mathrm{TV}}}
\newcommand{\adet}{\alpha_{\mathrm{det}}}
\newcommand{\Conf}{\mathrm{conf}}
\newcommand{\Simple}{\mathrm{simple}}

\title{Finite-size ignition laws near a vanishing branching deficit\\
on near-regular simple graphs}
\author{Aleksandr Rodionov\\
Yandex, Saint Petersburg, Russia\\
\href{https://orcid.org/0009-0006-6710-923X}{ORCID: 0009-0006-6710-923X}\\
\href{mailto:rexarrior@yandex.ru}{\texttt{rexarrior@yandex.ru}}}
\date{Preprint --- 28 October 2026}
\hypersetup{pdftitle={Finite-size ignition laws near a vanishing branching deficit on near-regular simple graphs},pdfauthor={Aleksandr Rodionov},pdfsubject={Bootstrap percolation on prescribed-degree simple graphs},pdfkeywords={bootstrap percolation, critical window, random graphs, Brownian motion}}
\begin{document}
\maketitle
\begin{abstract}
We study heterogeneous bootstrap percolation on uniform simple graphs with
prescribed degrees of mean $d\to\infty$ and deviations at most $C\sqrt d$.
Thresholds are independently one with probability $(1-c/d)/d$ and two
otherwise; seeds form an independent uniform set of prescribed size $a$.
Let $T_a$ be the terminal number of active vertices and $b$ be one minus
the mean number of threshold-one descendants of a reached vertex.
Let $K\sim d^2$ be the quadratic response coefficient. If $b>0$ and
$nb^3/K\to\infty$, we prove, uniformly over seed counts $a$,
\[
\Pp(T_a\ge n/2)=\Phi\!\left(\frac{a-n\adet}{w}\right)+o(1),
\qquad w=\sqrt{nb/K},
\]
where $\adet$ is the exact empirical response-barrier center and $\Phi$
is the standard normal distribution function. Put $a_0=(n/K)^{1/3}$.
If $d^4/n\to0$, $a/a_0\to x>0$ and $\lambda=ba_0\to\ell\in\mathbb R$,
then $\Pp(T_a\ge n/2)\to p(x,\ell)$, the one-sided Brownian parabolic
survival probability. A sequential absolute-error law joins these regimes.
A graphical counterexample shows that replacing the
exact center by its leading term can cause limiting probability error
$1/2$. The proof balances adaptive early-exploration transfer against late
concentration at a tunable physical horizon, despite exponentially rare
simplicity. We also prove quantitative and mesoscopic-seed refinements in
explicitly narrower regimes. The only external graph-enumeration input is
a stated reduction of McKay--Wormald's estimate; the probabilistic estimates
for the cascade argument are proved in the text and appendices.
\end{abstract}

\clearpage
\setcounter{tocdepth}{1}
\begingroup
\small
\tableofcontents
\endgroup
\medskip

\section{Introduction and contribution}
The excess degree determines the initial propagation of threshold-one
activation. As its branching deficit vanishes, the ignition bottleneck
approaches time zero, and a fixed-deficit approximation can miss both its
random minimum and its center at the critical-window scale. We resolve
this moving bottleneck for prescribed near-regular simple graphs and
connect it to the terminal event $T_a\ge n/2$.

The two limiting laws describe the same exploration queue at different
scales. When the normalized deficit $\lambda$ is large and positive,
the deterministic bottleneck lies far from the initial time on the
Brownian scale. Noise changes little across its narrow neighborhood, so a
single Gaussian height determines whether the queue survives. When
$\lambda$ stays bounded, the bottleneck and its fluctuations occupy
comparable time scales: survival depends on the minimum of a Brownian
path with parabolic drift. The exact center aligns these descriptions
when deterministic corrections are visible at the window scale.

\subsection{Related work and contribution}
Response-based centering and Gaussian branch selection are established
features of bootstrap and core processes. Janson, \L uczak, Turova and
Vallier obtain a Gaussian window with refined Poisson-response centering
in $G(n,p)$ \citep[Eq.~(3.16), Theorem~3.6 and Lemma~9.3]{fbnov-jltv}.
Janson and \L uczak prove Gaussian $k$-core critical probabilities for
prescribed-degree graphs, including the uniform simple graph
\citep[Theorem~3.5 and Section~7]{fbnov-jl}.
For fixed regular degree and a common threshold, Janson also describes
the critical window and proves a central limit theorem for the number
of external seeds needed for complete activation in the configuration
multigraph model \citep[Remark~5.4 and Theorem~5.6]{fbnov-janson}.
These locators refer to arXiv:0804.1656v1. This is a direct Gaussian
antecedent; its precise simple-graph counterpart is not established
in that version.
Ballen and Guha treat nonuniform thresholds, fixed seed cardinality and
a full mixed response under a fixed positive branching gap and relative
seed margin \citep[Definition~1 and Theorem~1]{fbnov-bg}.
Amini gives prescribed-degree response and cascade criteria
\citep[Theorems~4 and~5]{fbnov-amini}; Amini, Bayraktar and
Chakraborty use active-end/random-end pairing and prove stopped functional
and stable-root terminal central limit theorems
\citep[Section~1.2 and Theorems~1.3, 1.5]{fbnov-abc}.
Bootstrap/core duality \citep[Section~5]{fbnov-janson} and positive-quadratic
Brownian geometry in finite-size core and epidemic scaling are also known
\citep[Proposition~2.2]{fbnov-dm}
\citep[Eq.~(1.22) and Theorems~1.5--1.6]{fbnov-chmy}.

Balogh--Pittel is a close fixed-degree antecedent. For regular degree $d$
and common threshold $2\le k<d-1$, Theorem~1 bounds the transition by
seed-density margins $\omega n^{-1/2}$; it does not give the Gaussian
profile within that window. Theorem~2 treats $k=d-1$, and Theorem~3
extends the conclusions to random initial sets with concentrated size.
Their main setup already allows a deterministic initial set
\citep[pp.~260--262]{fbnov-bp}. Their response-ratio barrier and deferred
half-edge pairing precede the present construction. Substituting a mixed
response into that barrier gives its regular algebraic counterpart,
but does not extend their probability theorems to our mixed, nonregular
arrays. Their simplicity transfer uses a positive limiting probability
\citep[p.~264]{fbnov-bp}; Proposition~4 is proved for fixed degree
and does not state uniform growing-degree estimates
\citep[pp.~273--275]{fbnov-bp}. Their late
completion argument uses Corollaries~5 and~7 with Lemma~6
\citep[pp.~275--277]{fbnov-bp}. Here the weaker terminal requirement of
an $o(n)$ inactive set suffices. These differences concern the published
estimates, not the possibility of adapting their method.

The contribution is the joint exact Gaussian window, Brownian transition
and terminal macroscopic law for the stated growing-degree simple-graph
arrays. The fixed unscaled degree-law or uniform exponential-moment
hypotheses of the cited core and diffusion limit theorems do not cover
these arrays. The contact approaches the absorbing boundary, its raw
fluctuations are $o(\sqrt n)$, and simplicity has probability
$\exp[-d^2/4+O(d)]$. We address this conditioning by transferring a short
physical transcript and controlling the later evolution separately.
The sole enumeration input is the stated consequence of
McKay--Wormald's Lemma~5.1, explicitly reduced for residual sequences
including zeros \citep{mckay}.

The positive-deficit theorem covers every admissible
$\Lambda=nb^3/K\to\infty$, however slowly. The envelope itself implies
$\Lambda\le C_1n/d^5$, so no extra degree bound is needed.
The bounded-deficit and sequential laws hold under $d^4/n\to0$; the
latter retains one fixed finite lower bound on normalized deficit.
For $d\asymp n^\kappa$, $1/5<\kappa<1/4$, the envelope forces
$\lambda\to0$, leaving no Gaussian layer. An admissible tuning
$b\sim\beta d^{-2}>0$ has its Gaussian window when $n/d^8\to\infty$.
The sequential law follows from the two regimes, Brownian consistency
and monotone tails.

For the same threshold mixture, \citet{article1} proves threshold bounds
with a fixed positive excess-degree margin and fixed relative seed
margins, and derives deterministic finite-degree and higher-moment
corrections \citep[Section~6]{article1}. That work explicitly excludes
a shrinking random critical window. Here we determine the exactly
centered Gaussian probability profile, the Brownian crossover and their
sequential terminal law, including admissible vanishing-deficit regimes.
There is also a degree-range improvement for the macro event at a fixed
positive margin $c+1-v\ge\kappa>0$: here $b\asymp d^{-1}$ and
$\Lambda\asymp n/d^5$, so B gives the zero--one transition for
$T_a\ge n/2$ at fixed relative seed margins when $d=o(n^{1/5})$.
The prescribed-degree theorem of the earlier preprint assumes
$d=o(n^{1/7})$ \citep[Theorem~1]{article1}. The earlier theorem additionally
provides terminal-size bounds: $O(n/d^3)$ active vertices below threshold
and an explicit
$o(n)$ inactive remainder above. Our range comparison concerns the macro event.

The Brownian survival function is a classical one-sided parabolic-maximum
distribution \citep[Section~2]{groeneboom2010}
\citep[Theorems~1.1 and~3.1]{glt2013}. Appendix~\ref{app:seed} gives the
exact scaling correspondence for every real normalized deficit and a
self-contained radial representation of its boundary derivative. Neither
the distribution nor that coefficient is claimed as new.

\paragraph{Scope and route of proof.}
The laws are annealed and concern $T_a\ge n/2$ with fixed seed count and
iid thresholds. Their error is absolute; Theorem~\ref{thm:meso} gives a
relative mesoscopic refinement in a narrower regime.
Fixed-positive-integer seed terminal laws,
quenched laws and full activation are outside the results. The boundary
$n=O(d^4)$ limits this proof route, not necessarily the law itself.
Section~\ref{sec:model} states the main results.
Section~\ref{sec:realizability} records realizability; the full arithmetic
construction is in Appendix~\ref{sec:arith}. Sections~\ref{sec:det}--\ref{sec:range}
develop the common barrier, exploration, concentration and transfer tools.
Section~\ref{sec:X} proves X; Section~\ref{sec:quant} supplies the
quantitative comparisons used to complete B/C in
Section~\ref{sec:bridge}. Only then does Section~\ref{sec:S} give the
narrower tuned terminal-size and rate refinements.
Section~\ref{sec:ledger} records the dependencies. The appendices provide
the small-seed results, elementary probability arguments, arithmetic
construction and higher-order moment comparisons.

\section{Empirical model and the regime statements}\label{sec:model}
For each $n$ let $d_1,\ldots,d_n$ be deterministic nonnegative integers, with
even sum $M=nd$, and suppose
\begin{equation}\label{eq:envelope}
d\longrightarrow\infty,\qquad |d_i-d|\le C\sqrt d,\qquad
\pi=\frac{1-c/d}{d},\quad c>0,
\end{equation}
where $c,C$ are fixed. The exact mean $d$ may be nonintegral. All asymptotic
statements concern sufficiently large $n$; in particular $0<\pi<1$ and
$\delta:=\min_i d_i\ge4$. Write $\Delta=\max_i d_i$ and let $D$ denote the
empirical degree of a uniformly selected vertex. Set
\[
v=\Var(D)/d,\quad \mu_j=\E(D-d)^j,\quad
(k)_j=k(k-1)\cdots(k-j+1).
\]
Assume the degree vector is graphical. On a uniform simple graph with these degrees, choose independent thresholds
$R_i$ with $\Pp(R_i=1)=\pi$ and $\Pp(R_i=2)=1-\pi$. Independently choose a
uniform subset $A_0$ of $a$ seeds, where $a$ is deterministic. Starting with
$A_0$, repeatedly activate every vertex with at least $R_i$ active neighbors.
The terminal active count is $T$; we also write $T_a=T$ when the initial
set has size $a$. The terminal set is the least stable set
containing $A_0$, so the order of valid activations does not matter. All
probabilities in this article average over the graph, thresholds, and seed set.

The degree-weighted response and its exact coefficients are
\begin{align}
B_k(z)&=(1-z)^{k-1}+(1-\pi)(k-1)z(1-z)^{k-2},\label{eq:Bk}\\
B(z)&=\E[DB_D(z)]/d,\qquad F(z)=1-B(z),\label{eq:F}\\
b&=1-F'(0)=1-\pi\E(D)_2/d,\qquad
K=F''(0)=(1-2\pi)\E(D)_3/d.\label{eq:bK}
\end{align}
For $0\le z<1$, $B(z)>0$. Define
\begin{equation}\label{eq:center}
\mathcal R(z)=\frac{z-F(z)}{1-F(z)},\qquad
\adet=\max_{0\le z\le1/2}\mathcal R(z).
\end{equation}
The exact empirical distribution is used in every coefficient and in this
maximum. Replacing it by a limiting degree law is not part of the definition.
The envelope implies $K/d^2=1+O(d^{-1})$ and $v=O(1)$, as also verified below.

Define the exact crossover scales
\begin{equation}\label{eq:Xscales}
a_0=(n/K)^{1/3},\qquad b_0=a_0^{-1},\qquad
u_0=(nK^2)^{-1/3}=b_0/K,\qquad G=da_0,
\end{equation}
and write $\lambda=b/b_0$. Let $\mathcal B$ be standard Brownian motion
and define
\begin{equation}\label{eq:Xlimit}
p(x,\lambda)=\Pp\left(\inf_{s\ge0}
\{x+\mathcal B(s)-\lambda s+s^2/2\}>0\right)\quad(x>0),
\qquad p(x,\lambda)=0\quad(x\le0).
\end{equation}
This is a one-sided time interval, with variance-one Brownian motion.

\begin{theorem}[B: full positive Gaussian layer]\label{thm:B}
Assume~\eqref{eq:envelope}, $b>0$, and $\Lambda=nb^3/K\to\infty$.
Set $w=\sqrt{nb/K}$, and let $\Phi$ denote the standard normal
cumulative distribution function. Then
\[
\sup_{a\in\{0,\ldots,n\}}
\left|\Pp(T_a\ge n/2)-\Phi((a-n\adet)/w)\right|\longrightarrow0.
\]
There is no additional degree-growth assumption: these hypotheses imply
$n/d^5\to\infty$, hence $d^4/n\to0$. The center is the exact maximum
in~\eqref{eq:center}. The compact-coordinate ledger in~\eqref{eq:Bledger}
is not a rate uniform over all possible speeds of $\Lambda\to\infty$.
\end{theorem}

\begin{theorem}[X: compact Brownian layer]\label{thm:X}
Assume~\eqref{eq:envelope} and $d^4/n\to0$. If
$\lambda\to\ell\in\mathbb R$ and $a/a_0\to x>0$, then
$\Pp(T_a\ge n/2)\to p(x,\ell)$, with $0<p(x,\ell)<1$.
If $|\lambda|$ is bounded, $a=o(a_0)$ gives a limit zero and
$a/a_0\to\infty$ gives a limit one. At $a=0$, $T_0=0$ exactly.
The underlying functional convergence is proved only on fixed compact
scaled time intervals; no growing-interval Brownian approximation is asserted.
\end{theorem}

\begin{corollary}[C: exact-centered sequential unification]\label{thm:C}
Assume~\eqref{eq:envelope}, $d^4/n\to0$, and $\lambda\ge-L_0$ for
one fixed finite $L_0$. With $\lambda_+=\max(\lambda,0)$, along every
such admissible deterministic array,
\[
\sup_{a\in\{0,\ldots,n\}}
\left|\Pp(T_a\ge n/2)-p\left(
\frac{a-n\adet}{a_0}+\frac{\lambda_+^2}{2},\lambda\right)\right|
\longrightarrow0.
\]
This is an absolute-error statement. It does not cover
$\lambda\to-\infty$ or supply relative error for fixed positive $a$.
\end{corollary}

The Gaussian part is not nonempty throughout $d=o(n^{1/4})$:
the envelope forces $\Lambda\le C_1 n/d^5$. For
$d\asymp n^\kappa$, $1/5<\kappa<1/4$, every admissible array instead
has $\lambda\to0$. Section~\ref{sec:examples} gives graphical examples
and distinguishes the proof-route obstruction at $n=O(d^4)$ from a
counterexample to the law. The exact-center counterexample is proved in
Section~\ref{sec:bridge}.

\begin{theorem}[Mesoscopic seeds; narrower graph regime]\label{thm:meso}
Assume~\eqref{eq:envelope}, $d\asymp\log^2n$ with fixed comparison
constants, and $|\lambda|\le L$ for one fixed finite $L$. If
$a_0/\log d\ll a\ll a_0$, then
\[
\Pp(T_a\ge n/2)=\kappa(\lambda)\frac{a}{a_0}(1+o(1)),
\]
where the finite positive continuous $\kappa$ and its radial Brownian
representation are proved in Appendix~\ref{app:seed}. If
$\lambda\to\ell$, one may replace $\kappa(\lambda)$ by $\kappa(\ell)$.
This relative law is not extended to the full degree range of Theorem~\ref{thm:X}.
\end{theorem}

Table~\ref{tab:regimes} summarizes the regimes. The proof dependencies
and further quantitative assertions are recorded in
Table~\ref{tab:dependencies}.
\begin{table}[htbp]
\centering
\small
\setlength{\tabcolsep}{4pt}
\renewcommand{\arraystretch}{1.0}
\setlength{\abovecaptionskip}{0pt}
\caption{Regimes under the common model~\eqref{eq:envelope}.
The ignition probability laws concern $T_a\ge n/2$; S also gives
terminal-size bounds.}
\label{tab:regimes}
\begin{tabular}{@{}>{\raggedright\arraybackslash}p{.12\linewidth}>{\raggedright\arraybackslash}p{.43\linewidth}>{\raggedright\arraybackslash}p{.40\linewidth}@{}}
\toprule
Result & Additional assumptions & Conclusion \\
\midrule
B~(\ref{thm:B}) &
$b>0$, $nb^3/K\to\infty$ &
Exactly centered Gaussian law; absolute error uniform in $a$. \\
X~(\ref{thm:X}) &
$d^4/n\to0$, $\lambda\to\ell\in\mathbb R$, $a/a_0\to x>0$ &
Brownian limit $p(x,\ell)$; extreme seed scales are covered separately. \\
C~(\ref{thm:C}) &
$d^4/n\to0$, $\lambda\ge-L_0$ for fixed finite $L_0$ &
Exact-centered sequential law; absolute error uniform in $a$. \\
S~(\ref{thm:S}) &
$b\sim\beta d^{-2}$ for fixed $\beta>0$, $n/d^{10}\to\infty$.
Seed regimes as specified in Theorem~\ref{thm:S}. &
Stronger terminal-size bounds and a quantitative Gaussian budget. \\
Mesoscopic (\ref{thm:meso}) &
$d\asymp\log^2 n$, bounded $|\lambda|$, $a_0/\log d\ll a\ll a_0$ &
Relative law $\kappa(\lambda)a/a_0\,(1+o(1))$. \\
\bottomrule
\end{tabular}
\end{table}

\section{Realizability of the regimes}\label{sec:realizability}
The limit laws apply to any graphical degree array satisfying the stated
conditions. Their near-critical regimes are nonempty when $C^2>c+1$:
Proposition~\ref{prop:realize} constructs symmetric graphical empirical
laws on at most five degree values. For integer $d\to\infty$ and even
$n$ with $n/d\to\infty$, these include, for every fixed real $\beta$,
\[
b=\beta/d^2+O(1/(nd^2)).
\]
More generally, a target deficit $t$ can be matched to error
$O(1/(nd^2))$ whenever its target variance
$V_t=V_*-d^3t/(d-c)$ lies in $[v_-d,v_+d]$ for fixed
$0<v_-<c+1<v_+<C^2$, where $V_*=d+cd^2/(d-c)$.
The full lattice construction, parity
conditions and criterion for exact $b=0$ are proved in
Appendix~\ref{sec:arith}. Section~\ref{sec:examples} gives the resulting
Gaussian and Brownian examples.

\section{Uniform deterministic barrier and moment expansion}\label{sec:det}
This section analyzes the deterministic response barrier. The stochastic
ignition laws are proved in Sections~\ref{sec:X} and~\ref{sec:bridge}.
Define the additional exact coefficients
\[
T_3=F'''(0)=-(2-3\pi)\E(D)_4/d,\qquad
U_4=F''''(0)=(3-4\pi)\E(D)_5/d.
\]
Direct expansion of factorial moments gives
\begin{align}
b&=(c+1-v)/d+c(v-1)/d^2,\label{eq:bexact}\\
K&=(1-2/d+2c/d^2)
[d^2+(3v-3)d+2-3v+\mu_3/d].\label{eq:Kexact}
\end{align}
In particular $K/d^2=1+O(d^{-1})$, since $\mu_3=O(d^{3/2})$.

\begin{theorem}[Barrier for arbitrarily small actual deficit]\label{thm:D}
Under~\eqref{eq:envelope}, for sufficiently large $d$, if $b\le0$ then
$\adet=0$ and the unique maximizer is zero. If $b>0$, there is a unique
global maximizer $\widehat z$ on $[0,1)$, it lies below $1/(8\Delta)$, and
uniformly over all such positive deficits,
\begin{equation}\label{eq:Dleading}
\widehat z=\frac bK[1+O(b/d)],\qquad
\adet=\frac{b^2}{2K}[1+O(b/d)].
\end{equation}
No lower bound on $b$ is imposed. Writing $u=b/K$, $A=Ku^2/2$, one has
\begin{align}
\widehat z&=u+\frac{K-T_3}{2K}u^2+O(d^2u^3),\label{eq:location}\\
\adet&=A-\frac{T_3}{6}u^3+Au-A^2
 +\left(\frac{T_3^2}{8K}-\frac{U_4}{24}
 -\frac{5T_3}{12}+\frac{5K}{8}\right)u^4+O(d^5u^5).
\label{eq:Drefined}
\end{align}
All constants depend only on $c,C$.
\end{theorem}
\begin{proof}
We give finite inequalities before asymptotics to preserve uniformity when
$b$ is extremely small. More generally let the support lie in $[m,D_*]$,
where $m\ge8$, $\pi\le1/4$, and
\[
Q_*=(m-2)((1-\pi)(m-1)-1)\ge16D_*.
\]
These conditions hold under the envelope for all sufficiently large $d$.
For example, it suffices that $d\ge\max\{512,4C^2,c+1\}$.
Let $\ell_*=1/(8D_*)$.
Put $\mathcal H=B/(1-z)$, so $\mathcal R=1-\mathcal H^{-1}$. For one degree
with $a_k=(1-\pi)(k-1)-1$ and $b_k=1-\pi(k-1)$,
\[
\mathcal H_k=(1-z)^{k-3}(1+a_kz),\quad
\mathcal H'_k=(1-z)^{k-4}[b_k-(k-2)a_kz].
\]
Since $a_k>0$, $(k-2)a_k\ge Q_*$ and $b_k\le1$, this derivative is strictly
negative for every $z\ge\ell_*$, $z<1$. Positive degree weights preserve
the sign and exclude all later maxima.

Writing $y=1-z$, direct differentiation yields
\begin{align*}
F'_k&=(k-1)y^{k-3}[\pi y+(1-\pi)(k-2)z]\ge0,\\
F''_k&=(k-1)(k-2)y^{k-4}[(1-2\pi)y-(1-\pi)(k-3)z],\\
F'''_k&=-(k-1)_3y^{k-5}[(2-3\pi)y-(1-\pi)(k-4)z].
\end{align*}
For $0\le z\le\ell_*$, division of $F''_k(z)$ by its value at zero gives
the lower bound
\[
(1-D_*z)[1-(2D_*+1)z]
\ge\tfrac78(\tfrac34-1/(8D_*))>\tfrac12.
\]
The braces in $F'''_k$ are positive there, at least
$(5/4)(1-1/(8D_*))-1/8$. Consequently
\begin{equation}\label{eq:curvfinite}
K/2\le F''(z)\le K,\qquad |F'''(z)|\le L_*:=3D_*^3.
\end{equation}
The quotient derivative has numerator
\begin{equation}\label{eq:g}
\mathcal R'=g/B^2,\qquad g=B-yF',\quad g(0)=b,\quad g'=-yF''.
\end{equation}
Thus $g$ decreases strictly on $[0,\ell_*]$, and the preceding global sign
gives $g(\ell_*)<0$. If $b\le0$, the derivative is negative everywhere
after zero. If $b>0$, exactly one zero $\widehat z$ occurs before $\ell_*$,
and it is the unique global maximum. This proves the qualitative conclusions.

At that zero, $b=\int_0^{\widehat z}(1-t)F''(t)\dd t$. With $u=b/K$,
\[
u\le\widehat z\le2u/(1-\ell_*)<3u,
\quad 0\le K\widehat z-b\le(K+L_*)\widehat z^2/2.
\]
If $\eta_*=(1+L_*/K)u$, these imply
$0\le\widehat z/u-1\le9\eta_*/2$.
For $N(z)=z-F(z)$, Taylor's theorem gives
\[
bz-Kz^2/2\le N(z)\le bz-Kz^2/2+L_*z^3/6.
\]
Since $F\ge0$, $\mathcal R(u)\ge A$. At the maximum $N>0$, hence
$F(\widehat z)<\widehat z<3u$, and completing the quadratic square gives
\[
1\le\adet/A\le\frac{1+9\eta_*}{1-3u}.
\]
The denominator is positive: $b\le1$ and $K\ge21$ imply $u\le1/21$.
If $\eta_*\le1/12$, the relative error is at most $16\eta_*$.
Under the envelope $K\asymp d^2$, $L_*=O(d^3)$, and
$0<b\le(c+1)/d$, so $\eta_*=O(b/d)\to0$. This proves~\eqref{eq:Dleading}
with errors carrying the required powers of $b$ before any division by it.

The derivation of~\eqref{eq:location}--\eqref{eq:Drefined} is given in
Appendix~\ref{app:moments}.
\end{proof}

\section{The actual stopped exploration and its noise}\label{sec:explore}
\subsection{Uniform pairing, latent clocks, and terminal closure}
Temporarily work with a uniform pairing of the $M$ labelled half-edges.
Marks count incident matched edges, with multiplicity. On simplicity this is
exactly the neighbor rule of Section~\ref{sec:model}. Give all original ends
independent rate-one exponential clocks, independently of labels and seeds.
Reserve an active end using only revealed information; pair it to the first
ringing unreserved end, update the partner vertex, and reserve another active
end if available. Memorylessness gives a uniform partner and a waiting time
of rate equal to the number of unreserved ends, independent of the partner.
At stopping complete the remaining matching uniformly. Conditional on any
fixed seed and threshold assignment, induction on pair reveals gives the
uniform full matching. Thus matching and the assignments are independent.
Original clocks need not be independent of the completed matching.

Let $J$ be the real event count. With the initial reservation and a fictitious
last reservation at exhaustion, put $W(0)=M-1$ and $W=M-1-2J$ through the
stopping time $\tau$. Let $X_i(u)$ be the number of original clocks at vertex
$i$ that ring by time $u$, including any clocks on ends later removed manually.
Define for all $u$ the latent function
\begin{equation}\label{eq:latentH}
H(u)=\sum_{i\notin A_0}(d_i-X_i(u))\ind_{\{X_i(u)<R_i\}}.
\end{equation}
This equals the actual inactive-end count through $\tau$: an inactive end is
never manually reserved, and once activation occurs its original clock count
can never fall below its threshold. Afterwards~\eqref{eq:latentH} is only an
auxiliary function and is not asserted to describe a continuing cascade.
The real active queue excluding the reservation therefore satisfies
\begin{equation}\label{eq:stop}
Q=W-H\ge0\quad(u<\tau),\qquad Q(\tau)=-1.
\end{equation}
Indeed a zero queue can continue using its reserved end; a first hit at a
threshold-two vertex may then exhaust it. An active partner is unavailable
when $Q=0$, and can exhaust a queue of one, giving the same value $-1$.
Large-degree activation jumps are positive, so the final value cannot skip
below $-1$. If all ends are exhausted, the negative value is bookkeeping and
is never used as a rate. With no seeds define $\tau=0$ and $T=0$ directly.

At exhaustion every physically unmatched end belongs to an inactive vertex
with fewer than its threshold number of marks from active vertices. Uniform
completion adds edges only between inactive vertices and creates no activation.
Every explored activation is valid, and the resulting set is stable. It is
therefore exactly the bootstrap closure.

Write $z=1-e^{-u}$, $y=e^{-u}$, $\alpha=a/n$, and $m_{A_0}(u)=\E(H(u)\mid A_0)$.
One-vertex expectation and uniform sampling give the exact reference
\begin{equation}\label{eq:Da}
\bar H(u):=\E H(u)=(1-\alpha)ndyB(z),\qquad
D_a(u):=Me^{-2u}-\bar H(u)=ndy[\alpha+(1-\alpha)F(z)-z].
\end{equation}
For example the one-vertex expectation is
$ky^k+(1-\pi)k(k-1)zy^{k-1}=kyB_k(z)$, explaining the degree weight.

\subsection{Exact jumps and predictable bracket}
In the physical revealed-information filtration future clocks of removed
ends are not revealed. Immediately before an event let $A_u$ be the number
of active unreserved ends and let $I_{10},I_{20},I_{21}$ be inactive vertices
of threshold/mark type $(1,0),(2,0),(2,1)$, respectively. The jumps and rates are
\begin{center}
\begin{tabular}{lll}\toprule
Partner type&Rate&Queue jump\\\midrule
Active end&$A_u$ in total&$-2$\\
$i\in I_{20}$&$d_i$&$-1$\\
$i\in I_{10}$&$d_i$&$d_i-2$\\
$i\in I_{21}$&$d_i-1$&$d_i-3$\\\bottomrule
\end{tabular}
\end{center}
Activation with $r$ inactive ends removes all $r$ from $H$, whereas $W$
loses two; this verifies the two activation entries. Thus, before stopping,
\begin{align*}
\beta_Q&=-2A_u-\sum_{I_{20}}d_i+\sum_{I_{10}}d_i(d_i-2)
+\sum_{I_{21}}(d_i-1)(d_i-3),\\
M_Q(u)&=Q(u\wedge\tau)-Q(0)-\int_0^{u\wedge\tau}\beta_Q(s)\dd s
\end{align*}
is a martingale, with jumps bounded by $\Delta$, and exact bracket
\begin{equation}\label{eq:queuebracket}
\langle M_Q\rangle_u=\int_0^{u\wedge\tau}
\left[4A_s+\sum_{I_{20}}d_i+
\sum_{I_{10}}d_i(d_i-2)^2+
\sum_{I_{21}}(d_i-1)(d_i-3)^2\right]\dd s.
\end{equation}
This compensator follows by conditioning over a short time interval, summing
rate times squared jump, and taking the mesh limit; simultaneous events have
probability zero.

If $b>0$, $b\to0$, $a=O(nb^2/K)$, $d/n\to0$, and $r_b=O(b/K)$, the bracket
density through $r_b\wedge\tau$ is $nd^2[1+o_P(1)]$, uniformly in time.
Here is the check. The initial threshold-one sum over all vertices has mean
$\pi\sum_i d_i(d_i-2)^2\sim nd^2$ and variance $O(nd^5)$, hence relative
variance $O(d/n)$. Removing seeds costs at most $a\Delta^3=o(nd^2)$.
Loss of that sum by threshold-one hits through $r_b$ has normalized expected
size $O(dr_b)=O(b/d)$, and the $I_{21}$ sum has normalized expected size
$O(d^2r_b)=O(b)$. Monotonicity controls their whole intervals. The active
and $I_{20}$ terms are at most $5M=o(nd^2)$. Consequently
\begin{equation}\label{eq:bracketequiv}
\langle M_Q\rangle_u=nd^2(u\wedge\tau)[1+o_P(1)]
\quad(0\le u\le r_b),
\end{equation}
where a single random relative error bounds the entire interval. This is a
stopped statement; it does not replace $u\wedge\tau$ by $u$ on extinct paths.

The free-end count has a different, smaller martingale. Since $J$ has
preterminal intensity $W(s-)$,
\begin{equation}\label{eq:WM}
Z_W(u)=e^{2(u\wedge\tau)}W(u\wedge\tau)-(M-1)
=-2\int_0^{u\wedge\tau}e^{2s}(\dd J(s)-W(s-)\dd s).
\end{equation}
For $r\le U:=\log2$, its jumps and bracket obey
\begin{equation}\label{eq:Wbounds}
|\Delta Z_W|\le2e^{2r},\quad
\langle Z_W\rangle_r\le4e^{4r}Mr,\quad
\E\sup_{u\le r\wedge\tau}|W(u)-Me^{-2u}|^2\le CMr+C.
\end{equation}
The last assertion uses the proved Doob inequality in Appendix~\ref{app:prob}.
The exponential estimate there also gives
\begin{equation}\label{eq:Wexp}
\Pp\!\left(\sup_{u\le r\wedge\tau}|W(u)-Me^{-2u}|>e+1\right)
\le2\exp\!\left[-\frac{e^2}{2(4e^{4r}Mr+2e^{2r}e/3)}\right].
\end{equation}
All identities in this subsection use the physical filtration. Independent
latent-clock estimates next use a different filtration; no future clock is
added to the physical filtration to obtain~\eqref{eq:WM}.

\subsection{Weighted first-hit noise and the composition errors}
For the Gaussian-scale estimates in this subsection, assume $b>0$ and
$a=O(nb^2/K)$. For a nonseed degree-$k$ vertex let $Y=(k-X)\ind_{X<R}$. Define nested
indicators $B_i^*=\ind_{\{R_i=1,X_i\ge1\}}$ and
$C_{ji}=\ind_{\{R_i=2,X_i\ge j\}}$, $j=1,2$. Exactly
\begin{equation}\label{eq:loss}
k-Y=kB_i^*+C_{1i}+(k-1)C_{2i}.
\end{equation}
Conditional on the seed set, different vertices are independent. Within a
vertex these indicators are dependent; all variance bounds below apply the
squared triangle inequality to their separate families.
At $z=\Theta(t_b)$, $t_b=b/K$, $dz\to0$, their probabilities give
\begin{align}
\Var\Big(\sum_{i\notin A_0}d_iB_i^*\mid A_0\Big)
&\sim\pi z\sum_i d_i^3\sim nd^2z,\label{eq:leadvar}\\
\Var\Big(\sum_{i\notin A_0}[C_{1i}+(d_i-1)C_{2i}]\mid A_0\Big)
&\le Cn(dz+d^4z^2).\label{eq:restvar}
\end{align}
Indeed $p_{B,i}=\pi[1-(1-z)^{d_i}]=\pi d_i z[1+O(dz)]$,
$\Pp(X_i\ge1)\le d_i z$, and $\Pp(X_i\ge2)\le d_i^2z^2/2$.
Removing $a=o(n)$ indices changes the leading sum by relative $O(a/n)$,
uniformly over the chosen identities. The remainder-to-leading variance
ratio is $O(d^{-1}+b)$, and Cauchy--Schwarz bounds the mixed covariance.

Fixed-size sampling leaves a random sum of seed degrees. With
$g_i(u)=\E Y_i(u)$ its exact variance is
\begin{equation}\label{eq:seedvar}
\Var(m_{A_0}(u))=\frac{a(n-a)}{n-1}\Var_{\mathrm{emp}}(g_i(u)).
\end{equation}
On $z=O(t_b)$ the loss bound in~\eqref{eq:loss} implies
$0\le d_i-g_i\le C(dz+d^3z^2)=o(1)$, hence
$\Var(m_{A_0})=O(ad)=O(nb^2/d)$. Similarly, conditioning on thresholds
before averaging clocks produces mean variance
\[
\pi(1-\pi)\sum_{i\notin A_0}
[d_i(d_i-1)z(1-z)^{d_i-1}]^2=O(nd^3z^2)=O(nb^2/d).
\]
Both are negligible compared with the leading half-edge variance $nb$.
Threshold labels are already included in $B_i^*$; this calculation is not an
extra independent variance to add to it. Free-end noise by~\eqref{eq:Wbounds}
has variance order $nb/d$, also negligible. Dividing the leading half-edge
variance by the squared seed sensitivity $d^2$ identifies
\begin{equation}\label{eq:noise}
w_b=\sqrt{nb/K},\qquad \sigma_b=dw_b,\qquad
\Lambda=nb^3/K.
\end{equation}
The one-point limit and the stopped minimum comparison establishing its
interpretation as an ignition window are proved in Section~\ref{sec:quant}.

\section{Shared maximal and concentration estimates}\label{sec:concentration}
\begin{lemma}[Weighted nested indicators]\label{lem:nested}
Let independent random crossing times $V_i\in[0,\infty]$ have deterministic
real weights $v_i$. For $Z(t)=\sum_i v_i(\ind_{\{V_i\le t\}}-\Pp(V_i\le t))$,
\begin{align}
\E\sup_{t\le r}|Z(t)|^2&\le16\sum_i v_i^2\Pp(V_i\le r),\label{eq:nested}\\
\E\sup_{l\le t\le r}|Z(t)-Z(l)|^2
&\le16\sum_i v_i^2\Pp(l<V_i\le r).\label{eq:nestedinc}
\end{align}
\end{lemma}
\begin{proof}
Introduce independent copies $V_i'$. Jensen bounds the first left side by
the squared supremum of $\sum_i v_i(\ind_{V_i\le t}-\ind_{V_i'\le t})$ in
expectation. Swapping each original/ghost pair preserves the joint law and
allows independent Rademacher signs $\varepsilon_i$. The squared triangle
inequality bounds this by four times
$\E\sup_{t\le r}|\sum_i\varepsilon_i v_i\ind_{V_i\le t}|^2$.
Conditional on the crossing times, sort those at most $r$. Their signed sums
are partial sums of independent mean-zero variables; Doob's inequality gives
four times $\sum_{V_i\le r}v_i^2$. Averaging proves the factor sixteen.
Ties only remove some available partial sums. For the second inequality
replace times outside $(l,r]$ by infinity and apply the same argument to
increments from $l$.
\end{proof}

Apply this lemma separately to the three families in~\eqref{eq:loss}.
In the $u$ coordinate,
$p_{B,i}(u)=\pi(1-e^{-d_i u})\le Cu$,
$p_{C_1,i}(u)\le d_i u$, and $p_{C_2,i}(u)\le d_i^2u^2/2$. Therefore,
uniformly over every prescribed seed set,
\begin{equation}\label{eq:lossmax}
\E\!\left(\sup_{u\le r}|H(u)-m_{A_0}(u)|^2\mid A_0\right)
\le Cn(d^2r+dr+d^4r^2).
\end{equation}
The threshold-two centered part alone has bound $Cn(dr+d^4r^2)$.

For a bound uniform in the seed mean, partition vertices into the $q$
nonempty integer degree classes, where $q\le2C\sqrt d+2$. Let $p_k$ be their
empirical masses and $N_k$ the seed class counts. Without-replacement sampling
has
\[
\E N_k=ap_k,\qquad \sum_k\Var(N_k)
=\frac{a(n-a)}{n-1}\sum_k p_k(1-p_k)\le a.
\]
Since $m_{A_0}-\bar H=-\sum_k(N_k-ap_k)g_k(u)$ and $0\le g_k\le\Delta$,
Cauchy--Schwarz gives
\begin{equation}\label{eq:seeduniform}
\E\sup_{0\le u\le U}|m_{A_0}(u)-\bar H(u)|^2\le\Delta^2qa.
\end{equation}
This estimate depends on exact fixed seed size, not independent Bernoulli seeds.
For the smaller X seed scale another useful deterministic bound is
\begin{equation}\label{eq:seeddet}
\sup_{u\le r}|m_{A_0}(u)-\bar H(u)|
\le Ca\sqrt d+Ca(dr+d^3r^2).
\end{equation}
To prove it, decompose $g_i=d_i-(d_i-g_i)$, use
$|\sum_{i\in A_0}d_i-ad|\le Ca\sqrt d$, and bound both the sampled loss sum
and its population proportion by $Ca(dr+d^3r^2)$ using~\eqref{eq:loss}.

\begin{lemma}[Interval concentration]\label{lem:interval}
At fixed $u$, $z=1-e^{-u}$, put
$V(z)=\min\{Cn(d^2z+d^4z^2),n\Delta^2\}$. For any $e>0$,
\begin{align}
\Pp(|H-m_{A_0}|>e\mid A_0)&\le
2\exp\!\left[-\frac{e^2}{2(V(z)+\Delta e/3)}\right],\label{eq:Hbern}\\
\Pp(|m_{A_0}-\bar H|>e)&\le2\exp[-e^2/(2a\Delta^2)].\label{eq:seedexp}
\end{align}
At $a=0,n$ the second discrepancy is zero. On a deterministic interval
$[l,r]\subset[0,U]$ with $l<r$, choose the uniform grid with
$N=\lceil4(r-l)n\Delta^2/e\rceil$ subintervals. It has $N+1$ points
and mesh at most $e/(4n\Delta^2)$. For $l=r$, use the single point $l$. Except with
probability bounded by that many copies of~\eqref{eq:Hbern} and
\eqref{eq:seedexp}, evaluated with variance $V(z(r))$, plus~\eqref{eq:Wexp},
one has
\begin{equation}\label{eq:intervalQ}
|Q(u)-D_a(u)|\le4e+1\quad(l\le u\le r,\ u\le\tau).
\end{equation}
\end{lemma}
\begin{proof}
The loss identity and the elementary hit bounds show
$\Var(Y_i)\le3(\pi d_i^3z+d_i z+d_i^4z^2/2)$.
Conditional vertex independence and the Bernstein inequality proved in the
appendix give~\eqref{eq:Hbern}. For the seed mean, expose sampled labels in
random order. The sampled-sum Doob increment on draw $j$ equals
$(n-a)(g_{I_j}-\bar g_{j-1})/(n-j)$; its absolute value is at most $\Delta$.
The proved Azuma estimate gives~\eqref{eq:seedexp}.

Both $H$ and $\bar H$ are nonincreasing. Each one-vertex mean changes at rate
at most $d_i^2$, so $|\bar H'|\le n\Delta^2$.
If the grid deviations are at most $2e$, monotone bracketing between adjacent
grid times gives a deviation at most $9e/4$ everywhere, including across
jumps. Combine this with the free-end deviation $e+1$ from~\eqref{eq:Wexp}
and the identity through stopping to obtain~\eqref{eq:intervalQ}.
\end{proof}

\paragraph{A common dyadic bound.}
Write $r_D\le2C\sqrt d+2$ for the number of degree classes. Suppose
$0<h\le H\le\eta/d$ and $D_a(u)\ge cndKu^2$ on $[h,H]$, for fixed
$c,\eta>0$. For each dyadic interval $[v,\min(2v,H)]$, the loss,
seed-mean and stopped free-end maximal squared bounds sum to
$C[nd^2v+ndv+nd^4v^2+\Delta^2r_Da+1]$. Markov's inequality with a
fixed fraction of the margin $ndKv^2$ gives interval failure at most
\begin{equation}\label{eq:commonDyadic}
\mathcal D(v):=C\left[\frac{1+1/d}{nK^2v^3}
+\frac{d^2}{nK^2v^2}+\frac{r_Da}{n^2K^2v^4}
+\frac1{n^2d^2K^2v^4}\right].
\end{equation}
Otherwise a stop in that interval contradicts $Q(\tau)=-1$.
Summing $v=2^jh$ uses only the geometric series with exponents $2,3,4$;
thus $\Pp_{\Conf}(\tau>h,\tau\le H)\le C\mathcal D(h)$, with no
number-of-interval factor. If seed discrepancy is instead bounded
deterministically by $o(ndKh^2)$ through $2H$, absorb it into the margin
and omit the third term. Every comparison is through stopping.

\section{Enumeration and adaptive transfer}\label{sec:transfer}
The following reduced enumeration estimate follows from Lemma 5.1 of
\citet{mckay}. We state the estimate and give the reduction explicitly.

\paragraph{Enumeration input E.}
For nonnegative integer degrees $r_i$ of even total $L\to\infty$,
maximum $1\le\Delta_r=o(L^{1/3})$, and
$\nu=L^{-1}\sum_i(r_i)_2$, the probability $S_r$ that a uniform
labelled pairing is simple satisfies uniformly
\begin{equation}\label{eq:enuminput}
\log\Pp(S_r)=-\nu/2-\nu^2/4+O(\Delta_r^4/L).
\end{equation}
Zero degrees are permitted. The error constant is uniform over the
residual arrays in this domain, and remains symbolic in finite budgets.

To identify the precise consequence used, write $L_j=\sum_i(r_i)_j$.
The cited primary lemma, under $1\le\Delta_r^3=o(L)$, gives
\[
\log\Pp(S_r)=-\frac{L_2}{2L}-\frac{L_2^2}{4L^2}
-\frac{L_2^2L_3}{2L^4}+\frac{L_2^4}{4L^5}
+\frac{L_3^2}{6L^3}+O(\Delta_r^3/L).
\]
Since $L_j\le L\Delta_r^{j-1}$, each of the three additional terms is
$O(\Delta_r^4/L)$ and the original remainder is absorbed when
$\Delta_r\ge1$. Removing degree-zero vertices changes neither these
sums nor the pairing probability. This proves the short reduction to E;
the external enumeration proof itself is the stated borrowed dependency.

In our applications $L\sim nd$, $\Delta_r=O(d)$ and $d^3/n=o(1)$. Hence
$\Delta_r^3/L=O(d^2/n)=o(1)$ and its error is $O(d^3/n)=o(1)$, uniformly
including residual sequences with zero degrees. For the initial sequence,
writing $q=\sum_i(d_i)_2/M=O(d)$, this gives
\begin{equation}\label{eq:simpleprob}
\log s=-q/2-q^2/4+O(d^3/n),\qquad
s=\Pp(S)\ge e^{-C_sd^2}.
\end{equation}
In particular the ensemble is nonempty eventually. Conditioning a labelled
pairing on simplicity gives the uniform simple graph, since every such
graph has $\prod_i d_i!$ labelled pairings. The matching independence already
proved preserves the specified seed/threshold law after conditioning. It
does not preserve independent latent clocks.

\begin{lemma}[Prescribed-degree adaptive transcript transfer]\label{lem:TV}
An adaptive procedure reveals at most $m$ matched pairs, choosing each first
end from already revealed information and independent labels or innovations.
If $m=o(n/d)$ and $d^3/n=o(1)$, its complete capped transcript $\mathcal T_m$
satisfies
\begin{equation}\label{eq:TV}
\TV(\mathcal L(\mathcal T_m\mid S),\mathcal L(\mathcal T_m))
=O(md/n+d^3/n).
\end{equation}
The statement includes independent waiting-time innovations and procedures
that stop the physical cascade and continue only pair revelation to the cap.
\end{lemma}
Write the uniform constant in Input E as $A_E$: within its eventual range,
for the initial sequence and every residual sequence of total
$\ell\ge M/2$ and maximum at most $\Delta$, put $q_r=\ell^{-1}\sum_i(r_i)_2$. Then
\begin{equation}\label{eq:qEnum}
\left|\log\Pp(S_r)+q_r/2+q_r^2/4\right|
\le A_E\Delta^4/\ell.
\end{equation}
This just assigns a name to the constant in~\eqref{eq:enuminput}; it changes
neither the borrowed statement nor its source. The constant $A_E$ is kept symbolic here. The bounds below therefore
expose parameter dependence but do not provide a numerically evaluated
finite-$n$ certificate.

For a cap $m$, put $\ell=M-2m$ and assume $\ell\ge M/2$ and
$\ell>8\Delta^2+1$. Define
\begin{align}
e_m&=6m\Delta(1+\Delta)/M+3A_E\Delta^4/M,\nonumber\\
f_m&=\frac{m\Delta^2}{\ell-8\Delta^2},\qquad
\beta_m=\frac{m(\Delta+\Delta^2)}{\ell-1},\nonumber\\
\mathcal V_m&=\min\{1,\ \beta_m+e^{e_m}-1+f_m\}.
\label{eq:qTVfinite}
\end{align}
Then the capped transcript in Lemma~\ref{lem:TV} satisfies
\begin{equation}\label{eq:qTVresult}
\TV(\mathcal L(\mathcal T_m\mid S),\mathcal L(\mathcal T_m))
\le\mathcal V_m.
\end{equation}
\begin{proof}
Condition on all independent external labels and innovations. Sequential
uniform partners leave a uniform residual pairing, even with adaptive
first-end selection. For a simple history of exactly $m$ pairs, let $k_i$
be the used ends, $r_i=d_i-k_i$, and $L=M-2m=\ell$. Exact subtraction gives
\begin{equation}\label{eq:qdiff}
q-q_r=\frac{\sum_i k_i(2d_i-k_i-1)-2mq}{L},\qquad
|q-q_r|=O(m/n).
\end{equation}
Indeed $\sum_i k_i=2m$, $k_i\le\Delta$, $q,q_r\le\Delta$; the numerator
is at most $6m\Delta$, so $|q-q_r|\le12m\Delta/M$. Its sign need not
be positive. The difference of $-q/2-q^2/4$ and its residual version
costs at most $6m\Delta(1+\Delta)/M$, and the two enumeration remainders
cost at most $3A_E\Delta^4/M$, totaling $e_m$.

Residual simplicity must exclude copies of exposed edges. For such an edge
$uv$, switch $uv,xy$ to $ux,vy$ in a simple residual graph, with four
distinct endpoints and both new edges absent. Of its $\ell$ oriented edges
$xy$, fewer than $8\Delta^2$ are excluded by incidence to $u,v$, by
$x\in N(u)$ or by $y\in N(v)$. A reverse switch has at most $r_ur_v$
choices. Double counting yields
\[
\Pp(uv\in E(G_r)\mid S_r)\le\frac{r_ur_v}{\ell-8\Delta^2}
\le\frac{2\Delta^2}{\ell}
\]
eventually in the asymptotic regime; zero residual degrees give zero.
The union over exposed edges is bounded by $f_m$. Thus the conditioning
density on a simple history lies between
$\max\{0,e^{-e_m}(1-f_m)\}$ and $e^{e_m}$. On a nonsimple history it
is zero. At each reveal at most $\Delta+\Delta^2$ partners are forbidden,
from a pool of at least $\ell-1$; hence such a history has probability
at most $\beta_m$. Integrating half the absolute density deviation proves
\eqref{eq:qTVresult}. Since $\Delta=O(d)$ and $M=nd$,
\begin{equation}\label{eq:qTVorder}
\mathcal V_m\le C_E(md/n+d^3/n)
\end{equation}
when these asymptotic parameters are small, proving~\eqref{eq:TV}.
The density is then $1+O_E(md/n+d^3/n)$ on every simple history.
Fewer than $m$ reveals can be extended predictably to exactly $m$;
average over the external information to finish. Only physical transcripts,
not the unrestricted latent clock field, are transferred.
\end{proof}

Use independent $E_j\sim\operatorname{Exp}(1)$ and waiting times $E_j/W_j$,
$W_j=M-1-2j$, together with independent uniform partners. Include the next
waiting innovation after the final revealed pair in the cap. Whether the
next pairing precedes a prescribed time is then known without its partner.
After physical stopping only the matching transcript is extended, not the
cascade. Thinning a rate-$M$ Poisson process with probability $W/M$
dominates the physical count, whose intensity is at most $M$.

For a physical horizon $u_*$ and $m>Mu_*$, define
\begin{equation}\label{eq:qOverflow}
\mathcal O(m,Mu_*)=
\exp\{-m\log[m/(Mu_*)]+m-Mu_*\}.
\end{equation}
It is zero if $Mu_*=0$. The physical count has intensity at most $M$,
and Poisson exponential compensation, optimized at
$\theta=\log[m/(Mu_*)]$, proves this Chernoff bound. A cap $m\ge4Mu_*$
has overflow at most $e^{-km}$. The next waiting innovation makes overflow
recognizable without the next partner. An event recognizable by $u_*$
is replaced by a capped event only on overflow; its capped probability
costs one TV bound, and simple overflow costs at most its configuration
probability plus one TV bound. Thus a conservative physical-event transfer
bound is
\begin{equation}\label{eq:qEventTransfer}
3\mathcal V_m+2\mathcal O(m,Mu_*).
\end{equation}

\section{A tunable physical horizon and terminal passage}\label{sec:range}
This section proves the common extension needed for B and X. Put
\begin{equation}\label{eq:rangeparameters}
\begin{gathered}
J=d^2+\log(n+2),\qquad \xi=KJ/n,\qquad \theta=\xi^{1/6},\\
h_0=\theta/K,\qquad q_0=h_0/u_0=\theta/b_0,\qquad
\varepsilon=(J/n)^{1/4}.
\end{gathered}
\end{equation}
The subscript on $q_0$ distinguishes this scaled horizon from excess
degree and degree-class counts elsewhere in the text.

\begin{lemma}[Horizon feasibility]\label{lem:rangehorizon}
Under the envelope, $d^4/n\to0$ is equivalent to $KJ/n\to0$.
It permits a horizon $q_0u_0$ with small transcript TV and late vertex
exponent larger than both the simplicity cost and the mesh cost:
$q_0b_0\to0$ and $q_0^3/J\to\infty$. The choice above satisfies
\begin{equation}\label{eq:rangetuning}
q_0b_0=\theta\to0,\qquad
q_0^3/J=\xi^{-1/2}\to\infty,\qquad
\Psi:=nK^2h_0^3=\sqrt{nJ/K},\qquad \Psi/J=\xi^{-1/2}\to\infty,
\end{equation}
and $d^2h_0=o(1)$, $ndh_0\to\infty$,
$ndKh_0^2=Gq_0^2\to\infty$, $n\varepsilon^2=\sqrt{nJ}\gg J$.
The hypotheses of B imply all these facts without a separate degree bound.
\end{lemma}
\begin{proof}
The two required inequalities have a solution only if
$b_0^3J\to0$: multiply $(q_0b_0)^3\to0$ by
$J/q_0^3\to0$. Conversely substitution of~\eqref{eq:rangeparameters}
gives~\eqref{eq:rangetuning}, since $b_0^3=K/n$.
As $K\asymp d^2$ and $J\ge d^2$, $KJ/n\to0$ implies
$d^4/n\to0$. In the other direction $d^2=o(\sqrt n)$, so
$d^2\log(n+2)/n=o(\log(n+2)/\sqrt n)=o(1)$; hence $KJ/n\to0$.
In particular $a_0\to\infty$, $b_0\to0$, $q_0\to\infty$, and
$d^2h_0=(d^2/K)\theta\to0$. The scale identities give
$ndh_0=da_0^2q_0\to\infty$ and $ndKh_0^2=Gq_0^2\to\infty$.
Also $J/n\to0$ and $n\varepsilon^2=\sqrt{nJ}\gg J$.

For B, the envelope identity~\eqref{eq:bexact} gives
$0<b\le(c+1)/d$. Therefore
$\Lambda=nb^3/K\le C_1n/d^5$. Its divergence forces
$n/d^5\to\infty$, which implies $d^4/n\to0$.
At $u_L=-\log(1-16b/K)$, $d^2u_L=O(b)=o(1)$.
Thus $h_*:=\max(u_L,h_0)$ has $d^2h_*=O(b+\theta)=o(1)$.
The cap $m=\lceil4Mh_*\rceil$ satisfies
$m/n=O(dh_*+1/n)=o(1)$ and
$md/n=O(d^2h_*+d/n)=o(1)$. Its residual total is at least $M/2$
and greater than $8\Delta^2+1$ eventually. Input E has domain cost
$\Delta^3/(M-2m)=O(d^2/n)=o(1)$ and error $O(d^3/n)=o(1)$.
These statements also hold for X with $h_*=h_0$.
\end{proof}

\begin{lemma}[Common late passage and terminal equivalence]
\label{lem:rangefinal}
Assume $d^4/n\to0$. Consider either X comparison sizes
$a\le Xa_0$, $|\lambda|\le L$, $h_*=h_0$, or B comparison sizes
$a\le2nb^2/K$, $b>0$, $\Lambda\to\infty$,
$h_* =\max(u_L,h_0)$. Set $E_* =\{\tau>h_*\}$ and
$H_* =\Psi$ in X, $H_* =\max(\Psi,\Lambda)$ in B. Then
\begin{align}
\Pp_{\Simple}(E_*,\tau\le U)
&\le e^{C_sd^2}\{Cn^2e^{-kH_*}+Cd e^{-kn}\}=o(1),
\label{eq:rangelate}\\
\Pp_{\Simple}(T_a\ge n/2)-\Pp_{\Simple}(E_*)&\to0.
\label{eq:rangeequiv}
\end{align}
On late survival, outside an error at most
$4\exp(C_sd^2-k\sqrt{nJ})$ and~\eqref{eq:rangelate},
\begin{equation}\label{eq:rangeterminal}
n-T_a\le Cn\{d2^{-d+C\sqrt d}+\varepsilon\}=o(n).
\end{equation}
The constants can depend on the fixed compact bounds. No particular
$O(n^{2/3})$ residue is asserted in this enlarged range.
\end{lemma}
\begin{proof}
For sufficiently small fixed $\eta>0$, the derivative bounds of
Theorem~\ref{thm:D} give $F''\ge K/2$ on $[0,\eta/d]$. Thus
$F(z)-z\ge-bz+Kz^2/4\ge Kz^2/8$ once $z\ge8b_+/K$.
For B this holds from $u_L$; for X it holds from any
$Ru_0$ with fixed $R\ge32(L+1)$, and hence from $h_0$ eventually.
Since $h_{a/n}=(F-z)+(a/n)B\ge F-z$, $z\ge u/2$ and $1-z\ge1/2$
on the small interval, the queue margin is
\begin{equation}\label{eq:rangemargin}
D_a(u)\ge ndKu^2/64\quad(h_*\le u\le\eta/d).
\end{equation}
Beyond $\eta/d$ it is at least $knd$ through $U=\log2$.
To see this without an ODE theorem, at $z=\eta/d$ the lower bound
is a fixed $q>0$. Monotonicity gives $F-z\ge q/2$ until $z=q/2$;
on $[q/2,1/2]$ use
$B(z)\le(1+\Delta)(1-q/2)^{\delta-2}=o(1)$. The small comparison
seed densities are $o(1)$, so the reference in~\eqref{eq:Da} has the
claimed margin. Changing between $z$ and $u$ alters fixed constants only.

On a dyadic interval $[h,\min(2h,\eta/d)]$ take
$e_h=\rho ndKh^2$, where $\rho>0$ is a sufficiently small fixed
fraction of~\eqref{eq:rangemargin}. The vertex, sampled-seed and
free-end denominators in Lemma~\ref{lem:interval} and~\eqref{eq:Wexp}
are, respectively, bounded by constant multiples of
\[
nd^2h+nd^4h^2,\qquad ad^2,\qquad ndh+ndKh^2.
\]
In the first, $\Delta e_h=O(nd^4h^2)$ is included, using $K\asymp d^2$.
Dividing $e_h^2$ gives the three distinct lower exponents
\begin{equation}\label{eq:Bexponents}
I_V\ge k\frac{nK^2h^3}{1+d^2h},\qquad
I_A\ge k\frac{n^2K^2h^4}{a},\qquad
I_W\ge k\frac{ndK^2h^3}{1+Kh}.
\end{equation}
If $a=0$, the seed discrepancy is identically zero.
For X, $I_A/I_V\ge k(nh/a)(1+d^2h)\ge k_Xq_0a_0\to\infty$.
For B, $a\le2nb^2/K$ and $h\ge16b/K$ give
$I_A/I_V\ge k(Kh/b^2)(1+d^2h)\ge k/b\to\infty$.
Also $I_W/I_V\ge kd$, since $(1+d^2h)/(1+Kh)$ stays between
fixed positive constants. This verifies rather than discards the large-seed
composition error in B.

The function $h^3/(1+d^2h)$ increases. At $h_*$ its denominator
tends to one, while $nK^2h_*^3\ge\Psi$, and in B it is also at
least $16^3\Lambda$. Thus all exponents are at least $kH_*$.
The smallest tolerance is at least a constant times $Gq_0^2\to\infty$,
so the final reservation correction one is absorbed. The interpolation
grid of Lemma~\ref{lem:interval} has $O(1+1/(dh))$ points on an
interval of length $O(h)$. Its maximum is bounded by
$C(1+1/(du_0))=O(da_0)<Cn$ eventually. There are $O(\log n)$
intervals: $u_0\asymp n^{-1/3}d^{-4/3}\gg n^{-2/3}$ in this range.
Hence $Cn^2$ is a conservative common prefactor.
On a good interval the positive margin excludes an actual stop by
$Q(\tau)=-1$. On the final interval use tolerance $\rho nd$, vertex
variance $n\Delta^2$, $a\le n$ in sampled-seed Azuma, and free-end
bracket $O(nd)$. Each exponent is at least $kn$, with $O(d)$ grid
points. Only these small configuration bad events are divided by
$s\ge e^{-C_sd^2}$, proving~\eqref{eq:rangelate}: both
$\Psi/J\to\infty$ and $n/J\to\infty$.

At $U$ use any terminal tolerance $e_T=\varepsilon_T nd>0$.
Conditional vertex Hoeffding and sampled-seed Azuma each give
$2e^{-kn\varepsilon_T^2}$, uniformly for $a\le n$. Thus
$\Pp(H(U)>\bar H(U)+2e_T)\le4e^{-kn\varepsilon_T^2}$.
The exact mean is
$\bar H(U)\le nd(1+\Delta)2^{-\delta}$. On real survival,
every inactive vertex has at least $\delta-1$ unpaired ends; hence its
count, and the eventual inactive count, is at most $H(U)/(\delta-1)$.
For general $\varepsilon_T$, this gives inactive count at most
$Cn[d2^{-d+C\sqrt d}+\varepsilon_T]$, with conditioned error
$4e^{C_sd^2-kn\varepsilon_T^2}$. Taking $\varepsilon_T=\varepsilon$
proves~\eqref{eq:rangeterminal}, since $n\varepsilon^2=\sqrt{nJ}\gg J$. Latent independence is used
only before conditioning, not conditional on simplicity or on survival.

In the other direction a stop by $h_*$ without cap overflow has
$T_a\le a+m=o(n)$: every nonseed activation consumed at least one
partner reveal. Uniform matching completion joins inactive ends only
and cannot restart activation. The cap overflow probability under the
simple law is at most configuration overflow plus transcript TV,
$\mathcal O_*+\mathcal V_m=o(1)$, by~\eqref{eq:qEventTransfer} or
the direct capped argument following Lemma~\ref{lem:TV}.
This early error is not divided by simplicity. The two terminal
directions prove~\eqref{eq:rangeequiv}.
\end{proof}

\section{Proof of the X theorem}\label{sec:X}
Throughout this section assume $d^4/n\to0$, bounded $|b/b_0|$, and
first $a/a_0\to x>0$ and $b/b_0\to\ell$. Use the horizon $h_0$ constructed in
Section~\ref{sec:range}. The exact scale identities are
\begin{equation}\label{eq:scaleidentities}
b_0=a_0^{-1},\quad Ku_0=b_0,\quad nu_0=a_0^2,\quad
nKu_0^2=a_0,\quad a_0u_0=K^{-1},\quad nK^2u_0^3=1.
\end{equation}
They give $a_0\to\infty$, $b_0\to0$, $d^2u_0\asymp b_0\to0$,
$du_0\to0$, and $G\to\infty$. For $a=O(a_0)$,
\eqref{eq:seeddet} through $2h_0$ gives
\begin{equation}\label{eq:Xseed}
G^{-1}\sup_{u\le2h_0}|m_{A_0}(u)-\bar H(u)|
\le C[d^{-1/2}+h_0+d^2h_0^2]=o(1),
\end{equation}
uniformly over seed identities. We do not require $d^3u_0\to0$.

\subsection{A compact functional Brownian limit with explicit proof}
Fix a finite $R$. Conditional on any seed set of size $O(a_0)$, define
\[
L_n(s)=G^{-1}\sum_{i\notin A_0}d_i[B_i^*(u_0s)-p_i(u_0s)],
\qquad 0\le s\le R,
\]
where $p_i(u)=\pi(1-e^{-d_i u})$. We prove convergence to standard Brownian
motion as a process, with continuous limiting paths.

The indicators at distinct vertices are independent defective one-shot
events. Their covariance is
\[
\operatorname{Cov}(B_i^*(u_0s),B_i^*(u_0t))
=p_i(u_0\min(s,t))-p_i(u_0s)p_i(u_0t).
\]
Uniformly on fixed compact intervals,
$p_i(u_0s)=\pi d_i u_0s[1+O(du_0R)]$ and $p_i=O(u_0)$.
Hence the normalized covariance tends uniformly to $\min(s,t)$:
\[
G^{-2}\sum_{i\notin A_0}d_i^2p_i(u_0\min(s,t))
=\frac{\pi\sum_{i\notin A_0}d_i^3}{nd^2}\min(s,t)[1+o(1)],
\]
with the prefactor tending to one. The product term is $O(u_0)$ after
normalization. Deleting $a=o(n)$ vertices changes the prefactor by $O(a/n)$,
uniformly. For every linear combination of finitely many coordinates, the
largest centered normalized vertex contribution is $O(\Delta/G)=o(1)$.
The characteristic-function proof in Appendix~\ref{app:limits} therefore
gives all finite-dimensional centered Gaussian limits with Brownian covariance.

Here is tightness without importing a functional limit theorem. In the
filtration observing only these latent first-event indicators, and without
revealing thresholds initially, their hazard is
\[
h_i(u)=\frac{p_i'(u)}{1-p_i(u)}
=\frac{\pi d_i e^{-d_i u}}{1-\pi+\pi e^{-d_i u}}\le C.
\]
Thus $M_i(u)=B_i^*(u)-\int_0^u(1-B_i^*(t))h_i(t)\dd t$ are independent
martingales. For $Z_i=B_i^*-p_i$, subtract the deterministic equation
$p_i'=(1-p_i)h_i$ to obtain
$Z_i=M_i-\int h_iZ_i$. Since $\E|Z_i(t)|\le2p_i(t)\le Ct$,
\begin{equation}\label{eq:Xmartcorrection}
\E\sup_{s\le R}|L_n(s)-\mathcal M_n(s)|
\le\frac{Cndu_0^2R^2}{G}=\frac{CR^2}{K}\to0,
\quad \mathcal M_n(s)=G^{-1}\sum_{i\notin A_0}d_iM_i(u_0s).
\end{equation}
Jumps of $\mathcal M_n$ are at most $j_n=\Delta/G\to0$, and its bracket
has deterministic density in the $s$ coordinate at most
$G^{-2}u_0\sum_i d_i^2C\le C$. It also obeys the more precise relation
\begin{equation}\label{eq:Xbracket}
\langle\mathcal M_n\rangle_s
=G^{-2}\sum_{i\notin A_0}d_i^2\int_0^{u_0s}(1-B_i^*(t))h_i(t)\dd t
=s+o_P(1),
\end{equation}
uniformly on $[0,R]$: $h_i(t)=\pi d_i[1+O(du_0R)]$, and the nonnegative
subtracted indicator term has normalized mean $O(u_0R^2)$.

For tightness use the exponential martingale inequality on each deterministic
mesh cell of length $\delta$. The bracket increase in a cell is at most
$C\delta$, so, conditional on its starting information,
\[
\Pp\!\left(\sup_{s\text{ in cell}}
|\mathcal M_n(s)-\mathcal M_n(\text{left endpoint})|>\varepsilon\right)
\le2\exp[-\varepsilon^2/(2(C\delta+j_n\varepsilon/3))].
\]
There are at most $R/\delta+1$ cells. If all such cell deviations are below
$\varepsilon$, the oscillation at distance at most $\delta$ is at most
$3\varepsilon$. Consequently its probability of a larger oscillation has
limiting upper bound at most
$2(R/\delta+1)e^{-\varepsilon^2/(2C\delta)}$, tending to zero as
$\delta\downarrow0$. The correction in~\eqref{eq:Xmartcorrection} is uniformly
negligible. The compactness and finite-dimensional argument proved in
Appendix~\ref{app:limits} now shows
\begin{equation}\label{eq:LFCLT}
L_n\Rightarrow\mathcal B\quad\hbox{on every fixed }[0,R],
\end{equation}
uniformly in the conditional seed set. This proof gives a functional limit
and its continuous-path property; bracket convergence alone was not used
as a substitute for tightness or identification.

The centered threshold-two part of~\eqref{eq:loss}, divided by $G$, has
maximal mean square at most
\[
C\frac{n(du_0+d^4u_0^2)}{d^2a_0^2}
=O(d^{-1}+d^2u_0)=o(1).
\]
The seed error is removed by~\eqref{eq:Xseed}, and the stopped free-end error
satisfies, by~\eqref{eq:Wbounds},
\begin{equation}\label{eq:XW}
G^{-1}\sup_{u\le Ru_0\wedge\tau}|W(u)-Me^{-2u}|
\xrightarrow{P}0.
\end{equation}
All centered label noise is already in the rare-event distribution $p_i$.
For orientation, conditioning on labels would produce an additional conditional
mean variance $O(nd^3u_0^2)$, which divided by $G^2$ is $O(du_0)=o(1)$;
it is not added a second time.

The exact response expansion is
$F(z)=(1-b)z+Kz^2/2+O(d^3z^3)$. In~\eqref{eq:Da} put $u=u_0s$ and
$z=u+O(u^2)$. Using~\eqref{eq:scaleidentities}, uniformly on $[0,R]$,
\begin{equation}\label{eq:Xdrift}
D_a(u_0s)/G=a/a_0-(b/b_0)s+s^2/2+o(1).
\end{equation}
The cubic error is $O(du_0R^3)$; the coordinate, multiplicative $y$, and
seed-depletion errors are $O(u_0)$ for fixed $R$. These all vanish. Thus the
fully latent process
\begin{equation}\label{eq:YFCLT}
Y_n(s)=[Me^{-2u_0s}-H(u_0s)]/G
\Rightarrow x+\mathcal B(s)-\ell s+s^2/2=:Z_{x,\ell}(s)
\end{equation}
on each compact interval under the unconditioned clock law. No unrestricted
latent process is transferred to the simple law later.

\subsection{Continuity of the actual stopping event}
Let $E_R=\{\tau>Ru_0\}$. Equation~\eqref{eq:XW}, the terminal value $-1$,
and $G\to\infty$ give a random $\varepsilon_n\to0$ in probability such that
\begin{equation}\label{eq:Xsandwich}
\{\min_{s\le R}Y_n(s)>\varepsilon_n\}\subset E_R
\subset\{\min_{s\le R}Y_n(s)\ge-\varepsilon_n\}.
\end{equation}
For the left implication, a hypothetical stop contradicts its positive
latent value at $\tau$. For the right one, real survival makes the queue
nonnegative throughout the interval. Brownian motion with the deterministic
drift here is absolutely continuous with respect to Brownian motion on
$[0,R]$; the deterministic-drift argument and the reflection principle are
proved in Appendix~\ref{app:Brownian}. For $x>0$ they imply that
$\min_{s\le R}Z_{x,\ell}(s)$ has no atom at zero. The minimum is a
continuous functional under uniform convergence to continuous paths.
Therefore~\eqref{eq:YFCLT}--\eqref{eq:Xsandwich} give
\begin{equation}\label{eq:Xcompactsurvival}
\Pp_\Conf(E_R)\longrightarrow
p_R(x,\ell):=\Pp(\min_{s\le R}Z_{x,\ell}(s)>0).
\end{equation}

Use $m=\lceil4Mh_0\rceil$. By~\eqref{eq:rangetuning},
\begin{equation}\label{eq:Xcap}
m=o(n/d),\qquad md/n=O(\theta+d/n)=o(1),\qquad m\to\infty.
\end{equation}
Poisson overflow is $e^{-km}=o(1)$. Lemma~\ref{lem:TV}, with the next
waiting innovation included, transfers the actual events $E_R$,
$E_*:=\{\tau>h_0\}$ and $E_R\setminus E_*$ at cost $o(1)$.
The simple overflow bound is configuration overflow plus TV; no simplicity
division of this early error is required. Thus
\eqref{eq:Xcompactsurvival} holds under the simple law as well.

\subsection{Bridge from a fixed compact interval to the growing horizon}
Suppose $|b/b_0|\le L$ and fix $R\ge32(L+1)$. The deterministic
margin calculation in Lemma~\ref{lem:rangefinal}, using $z\ge u/2$
and $z\ge8Lb_0/K$, gives for either sign of $b$
\begin{equation}\label{eq:Xmargin}
D_a(u)\ge c_0ndKu^2\quad(Ru_0\le u\le\eta/d).
\end{equation}
Apply~\eqref{eq:commonDyadic} to $[Ru_0,h_0]$. The deterministic
seed discrepancy~\eqref{eq:Xseed} is $o(G)$ through $2h_0$, so it is
absorbed by every margin, whose first size is $GR^2$. On a dyadic
interval starting at $h=su_0$ the remaining failure bound is
\begin{align}
\Pp(\hbox{loss or free-end deviation exceeds its budget})
&\le C\left[\frac1{nK^2h^3}+\frac1{ndK^2h^3}
+\frac{d^2}{nK^2h^2}\right]+\frac{C}{(ndKh^2)^2}\notag\\
&\le C[s^{-3}+d^{-1}s^{-3}+b_0s^{-2}]+o(1)s^{-4},
\quad s=h/u_0.\label{eq:Xdyadic}
\end{align}
Here $d^2u_0\asymp b_0$ and $ndKh^2=Gs^2$, with $G\to\infty$.
The geometric sum therefore gives
\begin{equation}\label{eq:Xbridge}
\Pp_\Conf(E_R\setminus E_*)\le CR^{-3}+Cb_0R^{-2}+o(1).
\end{equation}
The constant is independent of the number of intervals and of admissible
$R$. Transfer only this real event by~\eqref{eq:Xcap}, obtaining
$\limsup_n\Pp_\Simple(E_R\setminus E_*)\le CR^{-3}$; the latent
losses that bounded it need not be measurable in the capped transcript.

\subsection{Late passage and the final macroscopic event}
Lemma~\ref{lem:rangefinal} applies with $h_*=h_0$: its initial
vertex exponent is $k\Psi$, its terminal tolerance is
$\varepsilon nd$, and both late errors remain negligible after
simplicity division. A stop by $h_0$ without early cap overflow has
$T_a\le a+m=o(n)$; survival to $U$ gives $n-T_a=o(n)$.
Consequently its two-direction equivalence is
\begin{equation}\label{eq:Xfinalequiv}
\Pp_{\Simple}(T_a\ge n/2)-\Pp_{\Simple}(E_*)\longrightarrow0.
\end{equation}
The proof of that lemma uses the exact stopped queue and sampled-seed
concentration and is independent of the compact functional limit.

\subsection{Infinite horizon and the extreme scaled seed sizes}
For fixed $R\ge32(L+1)$, compact survival, the bridge, and terminal
equivalence give
\[
p_R(x,\ell)-CR^{-3}\le\liminf_n\Pp(T\ge n/2)
\le\limsup_n\Pp(T\ge n/2)\le p_R(x,\ell).
\]
Appendix~\ref{app:Brownian} proves $\mathcal B(s)=o(s^2)$ almost surely,
so $Z_{x,\ell}(s)\to\infty$ and its infimum is attained at finite time.
The compact no-atom result gives
$p_R\downarrow\Pp(\inf_{s\ge0}Z_{x,\ell}(s)>0)=p(x,\ell)$.
Letting $R\to\infty$ proves~\eqref{eq:Xlimit}, with the order of limits
kept explicit.

The function is increasing in $x$ and decreasing in $\lambda$.
It is less than one because at any fixed positive time the process has
positive probability of being negative. For positivity choose
$R>\lambda+1$. A finite Brownian path tube around the function $s\mapsto
\lambda s$, thin enough to preserve positivity, has positive probability;
the Brownian small-ball and deterministic-shift proofs are in the appendix.
Its endpoint is at least a positive height $h$. After $R$, future drift is
at least $(R-\lambda)t$, and independent future Brownian increments have
probability at least $1-e^{-2(R-\lambda)h}>0$ of never crossing $-h$, by
the exponential-martingale bound proved there. Thus $p>0$.

As $x\downarrow0$, finite-interval absolute continuity and the reflection
principle show that Brownian motion with the bounded smooth drift started
at zero is negative somewhere immediately almost surely, so $p(x,\lambda)
\to0$. As $x\to\infty$, the finite infimum of
$\mathcal B(s)-\lambda s+s^2/2$ gives $p(x,\lambda)\to1$.

For $a=o(a_0)$ compare by nested uniform seed sets with
$\lceil\varepsilon a_0\rceil$. Every subsequence has a further subsequence
with $b/b_0\to\lambda\in[-L,L]$, and its limiting upper probability is
at most $p(\varepsilon,-L)$. Let $\varepsilon\downarrow0$ to get zero.
For $a/a_0\to\infty$, compare from below with $\lfloor H a_0\rfloor$ for
fixed $H$, giving subsequential lower limit at least $p(H,L)$; let
$H\to\infty$ to get one. Only bounded comparison sizes enter the proof,
so no assumption $a=O(a_0)$ is silently imposed on the corollaries.
Exactly zero seeds remain absorbed. This completes Theorem~\ref{thm:X}.

\subsection{Nonvacuity and a limit of the Brownian function}
If $C^2>c+1$, Proposition~\ref{prop:realize} makes both stochastic regimes
nonvacuous. S uses target $\beta/d^2$. For X target
$t=\lambda(d^2/n)^{1/3}$ with integer $d\asymp\log^2n$ and even $n$.
The variance shift $d^3t/(d-c)=O(d^{8/3}n^{-1/3})=o(1)$ stays in the
construction's interval; its deficit error $O(1/(nd^2))$ is $o(b_0)$,
and actual $K/d^2\to1$ gives $b/b_0\to\lambda$. This works for $\lambda=0$
without exact arithmetic zero. If $C^2\le c+1$, the capacity bounds imply
$b/b_0\to+\infty$ in this logarithmic-degree range; finite $\lambda$ is
then impossible.

There is also the separate consistency identity
\begin{equation}\label{eq:Gaussianconsistency}
p(\lambda^2/2+t\sqrt\lambda,\lambda)\longrightarrow\Phi(t)
\quad(\lambda\to+\infty)
\end{equation}
for every fixed real $t$, by the quantitative bound~\eqref{eq:qConsistency}.
The underlying pathwise assertion is also retained: after $s=\lambda r$
and Brownian scaling, the penalized infimum is
$\inf_{r\ge0}[\mathcal B(r)+\lambda^{3/2}(r-1)^2/2]$.
It tends almost surely to $\mathcal B(1)$: continuity controls a shrinking
neighborhood of one, the penalty diverges on its compact complement,
and subquadratic Brownian growth makes the infimum over $r\ge2$ diverge.
Its normal limit has no atom at $-t$. This concerns the Brownian function
and interchanges no graph limits.
Section~\ref{sec:bridge} proves the graph Gaussian layer separately.

\section{Quantitative comparison tools and finite error budgets}\label{sec:quant}
\begin{lemma}[Gaussian bottleneck localization]\label{lem:Gloc}
Assume~\eqref{eq:envelope}, $b>0$, and $\Lambda=nb^3/K\to\infty$.
Put $w=\sqrt{nb/K}$, $x_a=(a-n\adet)/w$, and
$u_L=-\log(1-16b/K)$. For every fixed $X>0$,
\[
\sup_{\substack{a\in\{0,\ldots,n\}\\|x_a|\le X}}
\left|\Pp_{\Conf}(\tau>u_L)-\Phi(x_a)\right|\longrightarrow0.
\]
\end{lemma}
The deterministic queue profile has a quadratic minimum. Outside a
shrinking neighborhood its rise dominates the noise; inside, the leading
noise varies negligibly, leaving a single asymptotically normal height.
Section~\ref{subsec:Gloc} proves the lemma through a stopped-event
comparison and gives its finite error budget~\eqref{eq:qGearly}.
The capped transfer and late continuation in Section~\ref{sec:bridge}
then turn early survival into the terminal simple-graph law B.

This section derives its probability rates directly. The finite comparison
inequalities retain their displayed domain conditions. The eventual
$O(1/\log d)$ graph rate below is restricted to $d\asymp\log^2 n$
and bounded deficit; the tuned Gaussian terminal rate in
Section~\ref{sec:S} retains $b\sim\beta d^{-2}$ and $n/d^{10}\to\infty$.
Constants denoted $C$ may
depend on fixed $c,C$ from the degree envelope and on a stated compact
parameter set; they never depend on a positive lower bound for $b$.
The number of degree classes is denoted $r_D\le2C\sqrt d+2$ here, so that
the letter $q$ can be used for a time-mesh size.
The finite derivative bounds used below are~\eqref{eq:curvfinite},
with $D_*=\Delta$: $K/2\le F''\le K$ and
$|F'''|\le3\Delta^3$ on $[0,1/(8\Delta)]$. Their finite
hypotheses are stated in the proof of Theorem~\ref{thm:D} and hold
eventually under~\eqref{eq:envelope}.

\subsection{Deterministic error in probability coordinates}
For $b>0$, put $t_b=b/K$, $A=b^2/(2K)$,
$\eta_D=(1+3\Delta^3/K)t_b$. The finite barrier proof gives
\[
0\le\widehat z/t_b-1\le(9/2)\eta_D,\qquad
1\le\adet/A\le(1+9\eta_D)/(1-3t_b).
\]
If $\eta_D\le1/12$, $0\le\adet/A-1\le16\eta_D$.
Since $nA/w=\sqrt\Lambda/2$, substituting the quadratic center incurs
the Gaussian-coordinate error bounded by
\begin{equation}\label{eq:qCenterError}
16nA\eta_D/w=8\sqrt\Lambda\,\eta_D.
\end{equation}
This need not be small even though $\eta_D=O(b/d)$ tends to zero.

In crossover coordinates let $G=da_0$, $x=a/a_0$, $\lambda=b/b_0$.
For $u=u_0s$, $0\le s\le R$, under the finite derivative bounds~\eqref{eq:curvfinite} and
$u_0R\le1/(8\Delta)$, $|b|\le1$,
\begin{equation}\label{eq:qDriftR}
\sup_{s\le R}\left|D_a(u_0s)/G-[x-\lambda s+s^2/2]\right|
\le\mathcal D_R,
\end{equation}
where
\[
\mathcal D_R=u_0[4xR+2|\lambda|R^2+2R^3+xb_0R^2
 +(\Delta^3/K)R^3].
\]
Here $x\ge0$. To prove this, write $z=1-e^{-u_0s}$, $v=z/u_0$ and
$y=e^{-u_0s}$. The exact scaled reference equals
$y[x-\lambda v+v^2/2-xF(z)+\mathcal R_3]$.
The bounds $|v-s|\le u_0s^2/2$, $1-y\le u_0s$,
$F(z)\le2z+Kz^2/2$, and $|F'''|\le3\Delta^3$ give
$|\mathcal R_3|\le(\Delta^3/(2K))u_0R^3$.
The first three terms cost at most
$u_0[xR+(3/2)|\lambda|R^2+R^3]$, and seed depletion costs
$u_0[2xR+(xb_0/2)R^2]$. They fit in the displayed larger coefficients.

\subsection{Smooth replacement and a fourth-moment maximal lemma}
Define $\varphi(t)=10t^3-15t^4+6t^5$ on $[0,1]$, zero below and one above.
Its first and second derivatives have bounds 2 and 6, and its second
derivative is globally Lipschitz with constant 60. Second-order Taylor
remainder at smoothing width $\eta$ is at most $10\eta^{-3}|h|^3$.

\begin{lemma}[A scalar quantitative normal comparison]\label{lem:qScalar}
For independent centered $X_i$, $|X_i|\le j$, and total variance
$\theta\in[1/2,2]$, for $0<\eta\le1/4$,
\[
\sup_t|\Pp(\sum_iX_i\le t)-\Phi(t)|
\le C[j/\eta^3+\eta+|\theta-1|].
\]
For sufficiently small $j$, taking $\eta=j^{1/4}$ gives the bound
$C[j^{1/4}+|\theta-1|]$.
\end{lemma}
\begin{proof}
Replace summands one at a time by independent centered normals of matching
variance, using smoothed indicators built from $\varphi$.
The third absolute moments of the original summands sum to at most
$j\theta$. Each matching normal has standard deviation at most $j$,
so their third absolute moments sum to at most $2j\theta$.
First and second Taylor terms cancel in each replacement, leaving at
most $30j\theta/\eta^3$. A normal interval of width $\eta$ costs
$C\eta$ for $\theta\ge1/2$. Differentiating
$\Phi(t/\sqrt\theta)$ gives a derivative uniformly bounded for
$\theta\in[1/2,2]$, and hence the variance correction $C|\theta-1|$.
This proves the assertion, including uniformity in $t$.
\end{proof}

\begin{lemma}[Weighted fourth-moment maximum]\label{lem:qFourth}
For independent crossing times $V_i$, real deterministic weights $w_i$,
$p_i=\Pp(l<V_i\le r)$, and $V=\sum_iw_i^2p_i$,
\[
\E\sup_{l\le t\le r}
\left|\sum_iw_i\{\ind_{\{l<V_i\le t\}}-\Pp(l<V_i\le t)\}\right|^4
\le160[V^2+\sum_iw_i^4p_i].
\]
\end{lemma}
\begin{proof}
The $L^4$ Doob factor used here is proved as follows. On a finite grid,
the stopping argument of Lemma~\ref{lem:Doob} gives
$a\Pp(N^*\ge a)\le\E[|N_T|\ind_{\{N^*\ge a\}}]$.
Multiply by $4a^2$ and integrate to get
$\E(N^*)^4\le(4/3)\E[|N_T|(N^*)^3]$.
H\"older gives $\E(N^*)^4\le(4/3)^4\E|N_T|^4$;
dense grids and monotone convergence give the continuous-time version.

Ghost-sample Jensen and symmetrization by swapping each original/ghost
pair insert independent Rademacher signs. The fourth-power triangle
inequality costs at most 16. Conditional on the crossing times, order
the indices with $l<V_i\le r$. Their signed partial sums form a martingale;
the final Rademacher fourth moment is at most three times the squared sum
of squared weights. Doob contributes $(4/3)^4$. Finally independence gives
\[
\E(\sum_iw_i^2\ind_{\{l<V_i\le r\}})^2
\le V^2+\sum_iw_i^4p_i.
\]
The total coefficient $16(4/3)^4\cdot3=4096/27<152<160$ proves the lemma.
It is applied to each within-vertex dependent indicator family separately.
\end{proof}

For a $q$-coordinate positivity event, use a product of shifted copies of
$\varphi$ with width $\eta$. Every second mixed partial is bounded by
$6\eta^{-2}$. Every third mixed partial is bounded by $60\eta^{-3}$:
derivative patterns $3$, $2+1$, and $1+1+1$ cost at most 60, 12, and 8.
Thus replacement of independent centered vectors $X_i$ by matching Gaussian
vectors $N_i$ costs at most
\begin{equation}\label{eq:qVectorReplacement}
10\eta^{-3}\sum_i[\E\|X_i\|_1^3+\E\|N_i\|_1^3].
\end{equation}
Matching its total Gaussian covariance $\Sigma$ to another $\Sigma_0$
adds at most
\begin{equation}\label{eq:qGaussianInterpolation}
3q^2\|\Sigma-\Sigma_0\|_{\max}/\eta^2.
\end{equation}
For a proof interpolate the covariances linearly. Differentiating a
nonsingular Gaussian density and integrating twice by parts gives derivative
$\tfrac12\sum_{jk}(\Sigma-\Sigma_0)_{jk}\E\partial_{jk}f$.
The second-derivative bound proves~\eqref{eq:qGaussianInterpolation}.
Adding an independent Gaussian covariance $\epsilon I$ and then letting
$\epsilon\downarrow0$ removes singularity by bounded convergence.

\subsection{A tunable compact Brownian comparison}
Conditional on any seed set, the leading field is
\[
L_{A_0}(s)=G^{-1}\sum_{i\notin A_0}d_i[B_i^*(u_0s)-p_i(u_0s)],
\quad p_i(u)=\pi(1-e^{-d_i u}).
\]
Define
\begin{align*}
\zeta_R&=\left|\frac{\pi\sum_i d_i^3}{nd^2}-1\right|
       +\frac{\pi a\Delta^3}{nd^2},\qquad
\mathcal C_R=R\zeta_R+C\Delta u_0R^2,\\
\mathcal E_R&=C_0[R/d+d^2u_0R^2+r_Da/a_0^2+G^{-2}].
\end{align*}
The exact covariance is
\[
G^{-2}\sum_{i\notin A_0}d_i^2
[p_i(u_0\min(s,t))-p_i(u_0s)p_i(u_0t)].
\]
Its first term is $\rho_{A_0}\min(s,t)$ up to
$C\Delta u_0R^2$, where $\rho_{A_0}=\pi\sum_{i\notin A_0}d_i^3/(nd^2)$;
the product term is at most $Cu_0R^2$. Consequently the maximum covariance
error from Brownian motion is at most $\mathcal C_R$.
Exact mean cancellation gives
\[
\pi\E D^3/d^2=(1-c/d)[1+3\Var(D)/d^2+\mu_3/d^3],
\quad\zeta_R\le C_{c,C}(d^{-1}+a/n).
\]
The centered threshold-two remainder has normalized maximal squared size
$C(R/d+d^2u_0R^2)$; stopped free ends cost $C(R/d+G^{-2})$;
the exact seed-mean bound~\eqref{eq:seeduniform} costs $Cr_Da/a_0^2$.
Their combined supremum exceeds $\varepsilon$ with probability at most
$\mathcal E_R/\varepsilon^2$.

Take $q\ge1$, $h=R/q$, positive $\eta,v,\varepsilon$, and mesh times $jh$.
Assume the finite derivative bounds~\eqref{eq:curvfinite},
$u_0R\le1/(8\Delta)$, $|b|\le1$, $\rho_{A_0}\le2$ uniformly,
and $x>4\varepsilon+2\mathcal D_R$. Put
\begin{align*}
k_R&=4\varepsilon+2\mathcal D_R+4v+2h(|\lambda|+R)+2\eta,\\
J_R&=\int_0^R(s-\lambda)^2\dd s,\qquad
\mathcal A_R(k)=\min\{1,e^{J_R/2}
[2k\sqrt{2/(\pi R)}]^{1/2}\},\\
p_R(x,\lambda)&=\Pp(\min_{s\le R}
[x+\mathcal B(s)-\lambda s+s^2/2]>0).
\end{align*}
For every admissible cap $m>MRu_0$, the following explicit bound holds:
\begin{align}
|\Pp_{\Simple}(\tau>Ru_0)-p_R(x,\lambda)|
\le{}&3\mathcal V_m+2\mathcal O(m,MRu_0)
       +\mathcal E_R/\varepsilon^2\nonumber\\
&+Cq^3R/(a_0\eta^3)+3q^2\mathcal C_R/\eta^2
       +C(Rh+R/a_0^2)/v^4\nonumber\\
&+4q e^{-v^2/(2h)}+\mathcal A_R(k_R).
\label{eq:qCompactBound}
\end{align}
The right side can always be truncated at one.

Here is the complete comparison proof. The loss decomposition gives the
leading field plus the above residual. A positive latent infimum beyond
the residual and free-end error excludes stopping by evaluation at
$Q(\tau)=-1$. Survival implies a nonnegative actual queue at every mesh
point. This physical sandwich costs at most $2\varepsilon+\mathcal D_R$;
it never assigns a physical queue after extinction.
For each leading mesh vector,
$\E\|X_i\|_1^3\le2q^3(d_i/G)^3p_i(Ru_0)$.
The sum of second weights times $p_i$ is at most $2R$, and one extra
weight is at most $\Delta/G\le2/a_0$. Thus their total third moment is
$Cq^3R/a_0$. Matching Gaussian vectors obey the same bound, using
$\|z\|_1^3\le q^2\sum_j|z_j|^3$ and the Gaussian marginal third moment.
Equations~\eqref{eq:qVectorReplacement}--\eqref{eq:qGaussianInterpolation}
give the two Gaussian comparison terms.

On a mesh cell the leading crossing variance is at most $2h$ and its
fourth-weight sum at most $8h/a_0^2$. Lemma~\ref{lem:qFourth}, Markov,
and a union over cells give $C(Rh+R/a_0^2)/v^4$ for an oscillation larger
than $v$. Reflection bounds Brownian oscillations by
$4q e^{-v^2/(2h)}$; the Gaussian tail bound follows by optimizing the
increment moment $\E e^{t\mathcal B(h)}=e^{t^2h/2}$.
The drift oscillates by at most $h(|\lambda|+R)$, and product smoothing
moves each positivity threshold by at most $\eta$.

Finally the continuous drifted minimum lies in $[-k,k]$ with probability
at most $\mathcal A_R(k)$. The density in Lemma~\ref{lem:Brownianmin},
with derivative $s-\lambda$, has second moment $e^{J_R}$.
Reflection gives density at most $\sqrt{2/(\pi R)}$ for the minimum of
$x+\mathcal B$; Cauchy--Schwarz therefore gives the displayed $\mathcal A_R$.
All mesh and smoothing ambiguities fit in $k_R$. This continuous-minimum
bound avoids an insufficient union bound on mesh-coordinate densities.
Average over exact fixed-size seeds and transfer only the physical
survival event by~\eqref{eq:qEventTransfer}. This proves~\eqref{eq:qCompactBound}.

\subsection{Quantitative tails and the full bounded-layer rate}
For the graph rates in this subsection retain $d\asymp\log^2 n$.
For $|\lambda|\le L$, $a/a_0\le X$, $32(L+1)\le R\le S=d$,
the margin~\eqref{eq:Xmargin} and~\eqref{eq:commonDyadic} with
$h=Ru_0$, $H=Su_0$ give, using~\eqref{eq:scaleidentities},
\begin{equation}\label{eq:qXtail}
\Pp_{\Conf}(\tau>Ru_0,\ \tau\le Su_0)
\le C[R^{-3}+d^2u_0R^{-2}
 +(r_Da/a_0^2+G^{-2})R^{-4}].
\end{equation}
The seed term is quantitative, not a discarded deterministic $o(1)$.
Transfer this real event with $m=\lceil4MSu_0\rceil$ and
\eqref{eq:qEventTransfer}; then $\mathcal V_m=O_E(db_0+d^3/n)$.
For Brownian motion the same dyadic Doob bound, using drift at least
$s^2/4$ for $s\ge4(L+1)$, gives
\begin{equation}\label{eq:qBrownianTail}
0\le p_R(x,\lambda)-p(x,\lambda)\le C_LR^{-3}\quad(x\ge0).
\end{equation}
Indeed on $[h,2h]$, $\E\sup_{s\le2h}|\mathcal B(s)|^2\le8h$;
defeating $h^2/4$ costs $C/h^3$, summable geometrically.

For this logarithmic-degree rate, repeat the dyadic proof of
Lemma~\ref{lem:rangefinal} from $du_0$ rather than $h_0$.
Here $d^2(du_0)=O(db_0)=o(1)$, the initial vertex exponent is
$nK^2(du_0)^3=d^3$, and the seed exponent is larger; the same
$Cn^2$ grid prefactor applies. The final exponent is $cn$.
At $U$ take terminal tolerance $n^{2/3}d$; its Hoeffding and sampled-seed
exponents are $cn^{1/3}$, which beats $d^2$ in this restricted regime.
The cap has $md/n=O(db_0)=o(1)$ and a stop before it gives $T=o(n)$.
Thus the full separate late/terminal envelope is
\begin{equation}\label{eq:qXlate}
\mathcal L_n=e^{C_sd^2}
[Cn^2e^{-c_1d^3}+Cd e^{-c_2n}+4e^{-c_3n^{1/3}}].
\end{equation}
These are already small unconditional errors divided by simplicity;
the order-one early probability is transferred by TV. Early stop with
no overflow gives $T\le a+m<n/2$ eventually; late survival gives the
terminal inactive-end bound. Add~\eqref{eq:qCompactBound},
\eqref{eq:qXtail}, \eqref{eq:qBrownianTail}, \eqref{eq:qXlate}; their
joint transfer and overflow can conservatively be covered by
$10\mathcal V_m+8\mathcal O(m,MSu_0)$.

For a concrete rate, fix $X>0$, $|\lambda|\le L$, and fixed
logarithmic-degree constants; allow $8d^{-1/8}\le x=a/a_0\le X$. Take
\[
R=(\log d/1000)^{1/3},\quad q=\lceil R d^{1/8}\rceil,
\quad\varepsilon=\eta=d^{-1/8},\quad v=d^{-1/64},\quad S=d.
\]
Here $\mathcal D_R=O_{X,L}(u_0[R+R^2+dR^3])=o(d^{-1/8})$, so
$x>4\varepsilon+2\mathcal D_R$ holds throughout this shrinking range.
All other conditions and bounds use only upper bounds for $x,|\lambda|$:
$r_Da/a_0^2=r_Dx/a_0$ decreases with $x$, and the minimum
anti-concentration factor has no positive lower-height dependence.
The residual, replacement, covariance, and fourth-moment terms are
\begin{align*}
&O(Rd^{-3/4}+d^{1/4}d^2u_0R^2+d^{1/4}r_DX/a_0),\\
&O(R^4d^{3/4}/a_0),\\
&O(R^3d^{-1/2}+R^4d^{1/2}\Delta u_0),\\
&O(Rd^{-1/16}+Rd^{1/16}/a_0^2).
\end{align*}
Brownian oscillations cost at most
$CRd^{1/8}e^{-cd^{3/32}}$. The deterministic error is negligible compared
with $v$, so $k_R=O(v)$. Eventually $J_R\le R^3$, making
$\mathcal A_R(k_R)\le Cd^{1/2000-1/128}R^{-1/4}$.
All these terms, TV, overflow, and late errors are smaller than
$R^{-3}=1000/\log d$. Thus eventually along admissible arrays,
\begin{equation}\label{eq:seedabsolute}
\sup_{\substack{8d^{-1/8}\le a/a_0\le X\\|\lambda|\le L}}
|\Pp_{\Simple}(T_a\ge n/2)-p(a/a_0,\lambda)|
\le C_{X,L,E}/\log d
\end{equation}
In particular, on every fixed positive compact set of seed coordinates,
\begin{equation}\label{eq:qXrate}
|\Pp(T\ge n/2)-p(a/a_0,b/b_0)|\le C/\log d.
\end{equation}
Numerical evaluation of this bound requires the enumeration constant;
its growing-horizon density factor prevents automatic uniformity in growing $L$.

\subsection{Quantitative exact-center Gaussian localization}\label{subsec:Gloc}
This proof uses a scalar replacement and a local modulus, rather than a
functional CLT. Let $b>0$, $t_b=b/K$, $w=\sqrt{nt_b}$, $\sigma=dw$,
$\Lambda=nb^3/K$, and $L_*=16$. Assume the finite derivative bounds~\eqref{eq:curvfinite},
$L_*t_b\le1/(8\Delta)$, $\adet\le1/2$ and $1-L_*t_b\ge1/2$.
The exact double zero and two integrations give the explicit curvature
\begin{equation}\label{eq:qGcurvature}
D_{n\adet}(u_s)/\sigma\ge\kappa\sqrt\Lambda(s-\widehat s)^2,
\quad\kappa=1/16,\quad u_s=-\log(1-t_bs),\quad
\widehat s=\widehat z/t_b.
\end{equation}
For the actual coordinate $x=(a-n\adet)/w$, $|x|\le X$, the sensitivity
error is at most
\begin{equation}\label{eq:qGdrift}
\mathcal D_G=Xt_b[2L_*+bL_*^2/2],
\end{equation}
since $1-(1-z)B(z)\le z+F(z)\le2z+Kz^2/2$.

At the bottleneck define the actual, computable variance bounds
\begin{align}
p_i^*&=\pi[1-(1-\widehat z)^{d_i}],\nonumber\\
\theta_+&=\sigma^{-2}\sum_i d_i^2p_i^*(1-p_i^*),\qquad
\theta_- =\theta_+-a\Delta^2\max_i p_i^*/\sigma^2,\nonumber\\
\zeta_G&=\max\{|\theta_+-1|,|\theta_--1|\}.
\label{eq:qGvariance}
\end{align}
For every seed set the leading normalized point variance is between these
endpoints. If $[\theta_-,\theta_+]\subset[1/2,2]$, Lemma~\ref{lem:qScalar}
gives Kolmogorov error at most $C[(\Delta/\sigma)^{1/4}+\zeta_G]$.
Expanding $p_i^*$ at the actual $\widehat z$ and using exact mean
cancellation and~\eqref{eq:Dleading} gives
$\zeta_G\le C_{c,C}(d^{-1}+b/d+a/n)$.

Conditional on the seed set $A_0$, define for $0\le s\le L_*$
\begin{equation}\label{eq:Gleadingprocess}
\mathcal A_n(s)=\sum_{i\notin A_0}d_i[B_i^*(u_s)-p_i(s)],\qquad
p_i(s)=\pi[1-(1-t_bs)^{d_i}],\qquad
A_n(s)=\mathcal A_n(s)/\sigma.
\end{equation}
Here $B_i^*$ is the threshold-one indicator from~\eqref{eq:loss}, and
$u_s=-\log(1-t_bs)$ as above. The nested bounds give
\begin{equation}\label{eq:qGmaxmodulus}
\E\sup_{s\le L_*}|A_n(s)|^2\le32L_*,
\qquad
\E\sup_{|s-\widehat s|\le\delta}
|A_n(s)-A_n(\widehat s)|^2\le64\delta.
\end{equation}
For the first bound use
$\pi\sum_i d_i^3/(nd^2)\le2$. For the local constant, apply the increment
lemma on $[\widehat s,\widehat s+\delta]$, costing at most $32\delta$.
On the left, reverse the crossing times within
$[\widehat s-\delta,\widehat s]$: the increment at $\widehat s-v$ is the
negative of a nested sum of those crossings in the last $v$ units.
These reversed defective crossing times remain independent across vertices,
and the same derivative bound costs $32\delta$. The squared maximum over
both sides is at most their two squared maxima, giving 64. This avoids
silently enlarging the claimed constant by subtracting two global suprema.

All remaining normalized suprema, including the centered threshold-two
terms, exact sampled seed mean and stopped free-end/reservation error,
have a squared budget
\begin{equation}\label{eq:qGresidual}
\mathcal E_G=C_{L_*}[d^{-1}+b+r_Da/(nt_b)+\sigma^{-2}].
\end{equation}
Consequently their combined size exceeds $\nu$ with probability at most
$\mathcal E_G/\nu^2$. This follows by dividing
$Cn(dt_b+d^4t_b^2)$, $\Delta^2r_Da$, and $CMu_L+C$ by $\sigma^2$
and applying the squared triangle inequality. With
$\Delta/d\le2$, $K\ge d^2/16$, and $u_L\le\log2$, one conservative
choice is $C_{L_*}=10^5L_*^2$.

Take $0<\delta,\nu<1/4$ with the neighborhood inside $[0,L_*]$ and
\begin{equation}\label{eq:qGcondition}
\nu+\mathcal D_G\le\kappa\sqrt\Lambda\delta^2/4.
\end{equation}
Use $H=\kappa\sqrt\Lambda\delta^2/4$ in the global bound
\eqref{eq:qGmaxmodulus}. Its failure probability is
$512L_* /(\kappa^2\Lambda\delta^4)$.
On its good event the deterministic penalty outside the neighborhood
beats both leading noises and the residual/sensitivity errors when compared
with the bottleneck value. Inside, the local modulus exceeds $\nu$ with
probability at most $64\delta/\nu^2$. Thus the latent minimum differs
from $x+A_n(\widehat s)$ by at most a constant times $\nu+\mathcal D_G$,
except for the displayed failure events. A positive latent minimum excludes
stopping by evaluation at $Q(\tau)=-1$; survival implies a nonnegative
actual queue at the deterministic bottleneck. Only the stopped free-end
error enters these two comparisons. Scalar normal replacement and normal
anti-concentration now prove
\begin{align}
|\Pp_{\Conf}(\tau>u_L)-\Phi(x)|\le{}&
C[(\Delta/\sigma)^{1/4}+\zeta_G+\nu+\mathcal D_G]\nonumber\\
&+(\mathcal E_G+64\delta)/\nu^2
 +512L_* /(\kappa^2\Lambda\delta^4).
\label{eq:qGfinite}
\end{align}

Choose $\nu=\max\{\Lambda^{-1/13},\mathcal E_G^{1/3}\}$,
$\delta=\nu^3$. Then $\delta/\nu^2=\nu$,
$\mathcal E_G/\nu^2\le\nu$, and $1/(\Lambda\delta^4)\le\nu$.
Condition~\eqref{eq:qGcondition} is eventual because
$\sqrt\Lambda\nu^5\ge\Lambda^{3/26}\to\infty$ and
$\mathcal D_G/\nu\to0$: $t_b=O(d^{-3})$ while
$\nu\ge c d^{-1/3}$. The neighborhood remains internal since
$\widehat s\to1$. We obtain
\begin{equation}\label{eq:qGearly}
|\Pp_{\Conf}(\tau>u_L)-\Phi(x)|
\le\mathcal G_X:=C_X[\Lambda^{-1/13}+\mathcal E_G^{1/3}
 +w^{-1/4}+\zeta_G+\mathcal D_G].
\end{equation}
Under the hypotheses of Lemma~\ref{lem:Gloc}, the envelope gives
$b\le(c+1)/d$ and $d^3/n\to0$ (Lemma~\ref{lem:rangehorizon}).
The right side tends to zero uniformly for bounded $x$:
$w=\sqrt\Lambda/b\to\infty$ and
$r_Da/(nt_b)=O_X(r_Db+r_D/w)\to0$.
This proves Lemma~\ref{lem:Gloc}.
For simple graphs add~\eqref{eq:qEventTransfer} with
$m=\lceil4L_*Mt_b\rceil$; its exact horizon is $u_L\le2L_*t_b$,
so its cost is $O_E(b+d^3/n)$ plus the exact overflow.

\subsection{Quantitative consistency of the Brownian function}
Brownian scaling and completing the square give
\[
p(\lambda^2/2+t\sqrt\lambda,\lambda)=
\Pp\!\left(t+\inf_{r\ge0}
[\mathcal B(r)+\lambda^{3/2}(r-1)^2/2]>0\right).
\]
On $[0,2]$, the preceding localization proof has exact normal offset
$\mathcal B(1)$, curvature of order $\lambda^{3/2}$, and no residual.
Its global, local, and anti-concentration errors are
$C/(\lambda^3\delta^4)+C\delta/\nu^2+C\nu$.
For $r\ge2$, divide into $[2^j,2^{j+1}]$. Doob gives probability at most
$C\lambda^{-3}2^{-3j}$ that Brownian motion defeats half the quadratic
penalty. Also
$\Pp(|\mathcal B(1)|>\lambda^{3/2}/16)\le C\lambda^{-3}$.
Outside those failures the tail value exceeds the bottleneck value.
Taking $\nu=\lambda^{-3/13}$, $\delta=\nu^3$ therefore proves
\begin{equation}\label{eq:qConsistency}
\sup_{|t|\le X}|p(\lambda^2/2+t\sqrt\lambda,\lambda)-\Phi(t)|
\le C_X\lambda^{-3/13}.
\end{equation}
This concerns the Brownian function only. A final-cascade graph bridge
requires the separate proof that follows.

\section{The positive Gaussian bridge and corrected sequential law}\label{sec:bridge}
We now prove Theorem~\ref{thm:B} and Corollary~\ref{thm:C}. Throughout the
positive-layer argument $b>0$, and
$\Lambda\to\infty$. First keep the exact coordinate
$x=(a-n\adet)/w$ in a fixed bounded interval. By~\eqref{eq:Dleading},
$n\adet/w\ge\sqrt\Lambda/2$, so these sizes are positive and, eventually,
\begin{equation}\label{eq:Bseedsize}
a\le2nb^2/K,\qquad a=o(n).
\end{equation}
For instance $\adet\le2A$ and $\Lambda\ge16X^2$ give this conservative
bound. No assumption $a=O(a_0)$ is made.

\subsection{A summable real-event bridge to a growing early horizon}
Set $L=16$, $u_L=-\log(1-Lb/K)$ and
\[
h_* =\max\{u_L,h_0\},\qquad E_* =\{\tau>h_*\}.
\]
Both terms are $o(d^{-1})$, and $d^2h_*=O(b+\theta)=o(1)$.
Since $z(u_L)=Lb/K\ge8b/K$, the deterministic calculation in
Lemma~\ref{lem:rangefinal} gives
\begin{equation}\label{eq:Bmargin}
D_a(u)\ge ndKu^2/64\quad(u_L\le u\le\eta/d).
\end{equation}
If $h_*>u_L$, apply~\eqref{eq:commonDyadic} on $[u_L,h_*]$. Its
single-interval bound, with the exact sampled-seed term retained, is
\begin{equation}\label{eq:Bdyadic}
C\left[\frac{1+1/d}{nK^2h^3}+\frac{d^2}{nK^2h^2}
 +\frac{r_Da}{n^2K^2h^4}+\frac1{n^2d^2K^2h^4}\right].
\end{equation}
The geometric sum, $u_L\ge Lb/K$ and~\eqref{eq:Bseedsize} give
\begin{equation}\label{eq:Bsum}
\mathcal B_X=C_X\left[\frac{1+1/d}{L^3\Lambda}
 +\frac{(d^2/K)b}{L^2\Lambda}+\frac{r_Db}{L^4\Lambda}
 +\frac{b^2}{d^2L^4\Lambda^2}\right]\le C_X/\Lambda.
\end{equation}
Set $\mathcal B_X=0$ for an empty interval. Then
$\Pp_{\Conf}(\tau>u_L,\tau\le h_*)\le\mathcal B_X$.
Since $r_Db=O(d^{-1/2})$, the seed term is $O_X(r_Db/\Lambda)$,
even for diverging $a/a_0$, with no interval-count factor.

Use $m=\lceil4Mh_*\rceil$ and the capped transcript already proved.
It recognizes the two survival events and their difference, apart from
overflow. Write
\[
\mathcal O_* =\mathcal O(m,Mh_*)\le e^{-km},\qquad
\mathcal V_* =\mathcal V_m=O_E(b+\theta+d^3/n+d/n)=o(1).
\]
Indeed $md/n=O(d^2h_*+d/n)$, $m=o(n/d)$, and all residual enumeration
conditions and denominators hold eventually. Only real early events
are transferred. The unrestricted latent field used to derive the maximal
probabilities is not conditioned on simplicity.

\subsection{Late passage and the complete Gaussian ledger}
The envelope implication in Lemma~\ref{lem:rangehorizon} gives
$d^4/n\to0$ from $\Lambda\to\infty$. Thus the common late argument
in Lemma~\ref{lem:rangefinal} applies, including the full fixed-size
sampled-seed exponent. It gives
\begin{equation}\label{eq:Blate}
\Pp_{\Simple}(E_*,\tau\le U)
\le e^{C_sd^2}\{Cn^2e^{-k\max(\Lambda,\Psi)}+Cd e^{-kn}\}=o(1),
\end{equation}
and the terminal bound
\begin{equation}\label{eq:Bterminal}
n-T_a\le Cn\{d2^{-d+C\sqrt d}+\varepsilon\}=o(n)
\end{equation}
outside the stated late and terminal bad events. On early stopping,
$T_a\le a+m=o(n)$ outside simple-law overflow
$\mathcal O_*+\mathcal V_*$. No stronger terminal-size rate is inherited.

Lemma~\ref{lem:Gloc} gives early Gaussian survival, with the
budget~\eqref{eq:qGearly} valid for the full positive layer:
$w=\sqrt\Lambda/b\to\infty$ and, by
\eqref{eq:Bseedsize}, $r_Da/(nt_b)\le2r_Db=O(d^{-1/2})$.
Combining it with the real-event bridge, capped transfer, and both
directions of terminal passage proves, uniformly for bounded actual
$x=(a-n\adet)/w$,
\begin{align}
|\Pp(T_a\ge n/2)-\Phi(x)|\le{}&
\mathcal G_X+\mathcal B_X+10\mathcal V_*+8\mathcal O_*\nonumber\\
&+e^{C_sd^2}\{Cn^2e^{-k\max(\Lambda,\Psi)}
+Cd e^{-kn}+4e^{-k\sqrt{nJ}}\}.
\label{eq:Bledger}
\end{align}
Every term tends to zero even if $\Lambda$ diverges arbitrarily slowly.
The finite localization, derivative, denominator, and margin conditions
in~\eqref{eq:qTVfinite}--\eqref{eq:qGearly} precede use of this ledger;
the enumeration constant remains symbolic. The term $\Psi$, rather than
$\Lambda$ alone, makes the late simplicity division valid.

\subsection{All seed sizes, exact correction, and subsequences}
Let $f_n(a)=\Pp(T_a\ge n/2)$ and $y_n(a)=(a-n\adet)/w$.
Uniform seed sets coupled as prefixes of an independent permutation make
$f_n$ nondecreasing. In the positive layer the reference sizes
$\lfloor n\adet\pm Xw\rfloor$ are positive and below $n$ for every fixed
$X$ eventually, because $n\adet/w\to\infty$ and the center is $o(n)$.
Above $y=X$, compare from below with the positive reference; below
$y=-X$, compare from above with $\lceil n\adet-Xw\rceil$.
The bounded-window ledger and rounding $1/w\to0$ give
\[
\limsup_n\sup_{0\le a\le n}|f_n(a)-\Phi(y_n(a))|\le\Phi(-X).
\]
Let $X\to\infty$. This proves the all-seed assertion of
Theorem~\ref{thm:B} without assuming a CLT uniform in unbounded coordinates.
The same monotone endpoint argument applied to~\eqref{eq:qConsistency}
gives
\begin{equation}\label{eq:BfunctionUniform}
\sup_{t\in\mathbb R}|p(\lambda^2/2+t\sqrt\lambda,\lambda)-\Phi(t)|
\longrightarrow0\quad(\lambda\to+\infty).
\end{equation}

For bounded $\lambda$ on the positive side, the exact barrier gives
\begin{equation}\label{eq:BboundedCorrection}
0\le n\adet/a_0-\lambda^2/2\le C\lambda^3/(a_0d)\to0.
\end{equation}
On the nonpositive side $\adet=0$ exactly. Therefore
$x_{\mathrm{eff}}=(a-n\adet)/a_0+\lambda_+^2/2=a/a_0+o(1)$ uniformly
over sizes when $\lambda$ stays in a compact set.
The function $p$ is jointly continuous for positive $x$ and finite
$\lambda$: finite-horizon no-atom continuity follows from
Lemma~\ref{lem:Brownianmin}, and the uniform dyadic tail
\eqref{eq:qBrownianTail} removes the horizon. As $x\downarrow0$,
comparison with $p(x,-L)$ and immediate Brownian negative excursions
give zero uniformly on compact $\lambda$ sets. As $x\to\infty$,
comparison with $p(x,L)$ and the almost surely finite parabolic minimum
give one uniformly. Extend $p=0$ for nonpositive $x$.

For bounded scaled deficits, Theorem~\ref{thm:X} covers each subsequential
limit $a/a_0\to x$ with $x=0$, positive finite, or infinity. Together with
\eqref{eq:BboundedCorrection} and those continuity/tail facts it gives
the corrected approximation uniformly over every integer $a$.
Formally, any violating seed sequence has a subsequence with $\lambda$
convergent and then a further subsequence with $a/a_0$ convergent in
$[0,\infty]$; all three limits contradict the violation.

For $\lambda>0$, the exact identity $w=a_0\sqrt\lambda$ yields
\[
x_{\mathrm{eff}}=\lambda^2/2+y_n(a)\sqrt\lambda.
\]
The separately proved graph Gaussian law and
\eqref{eq:BfunctionUniform} therefore give the uniform corrected law
when $\lambda\to+\infty$. Finally, if $\lambda\ge-L_0$ is arbitrary,
every subsequence has a further finite-$\lambda$ subsequence or a
$+\infty$ subsequence. Both have just been proved. This proves
Corollary~\ref{thm:C}. At $a=0$, $x_{\mathrm{eff}}\le0$ by
$\adet\ge b^2/(2K)$ or by exact zero center; both probabilities are
exactly zero. This completes the sequential absolute-error proof.

\subsection{A strict counterexample to the uncorrected formula}
For an admissible $b\sim\beta/d^2$, $\beta>0$ and logarithmic degrees,
$T_3\sim-2d^3$. The proved barrier expansion gives
\begin{equation}\label{eq:BnaiveCorrection}
\adet-A=-T_3b^3/(6K^3)+b^3/(2K^2)+O(b^4/d^4)
\sim b^3/(3d^3)>0.
\end{equation}
The denominator term is smaller by $1/d$ and the remainder by $b/d$;
the positive leading correction is attained, not merely bounded above.
At $a=\lfloor n\adet\rfloor$ the true exact coordinate tends to zero,
so the graph probability tends to $1/2$. The naive Brownian argument has
coordinate
\[
\frac{a/a_0-\lambda^2/2}{\sqrt\lambda}
=\frac{n(\adet-A)}w+o(1)
\sim\frac{\beta^{5/2}\sqrt n}{3d^7}\longrightarrow+\infty.
\]
Equation~\eqref{eq:BfunctionUniform} gives $p(a/a_0,\lambda)\to1$,
so the uncorrected formula has limiting error $1/2$.
Proposition~\ref{prop:realize} supplies such graphical tuned examples
under $C^2>c+1$. The exact-center correction in Corollary~\ref{thm:C}
is therefore essential.

\section{A restricted quantitative tuned subregime}\label{sec:S}
\begin{theorem}[S: tuned quantitative subregime]\label{thm:S}
Assume~\eqref{eq:envelope}, $b\sim\beta/d^2$ for fixed $\beta>0$,
and $n/d^{10}\to\infty$. Put $A=b^2/(2K)$ and $w=\sqrt{nb/K}$.
For every fixed $\gamma\in(0,1)$ and every admissible deterministic
sequence of seed sizes,
\begin{align*}
a\le(1-\gamma)nA&\ \Longrightarrow\ T=O(n/d^4)
\quad\text{with probability tending to one},\\
a\ge(1+\gamma)nA&\ \Longrightarrow\ n-T=
O\bigl(n[d2^{-d+C\sqrt d}+n^{-1/3}]\bigr)
\quad\text{with probability tending to one}.
\end{align*}
For every fixed $x$, $a=\lfloor n\adet+xw\rfloor$ gives
$\Pp(T\ge n/2)\to\Phi(x)$, with the quantitative terminal
budget~\eqref{eq:qSrate}. The stronger terminal-size bounds and that
budget retain $n/d^{10}\to\infty$, although the probability-window
assertion alone follows from B under $n/d^8\to\infty$.
\end{theorem}

Throughout this section $b\sim\beta/d^2>0$ and $n/d^{10}\to\infty$.
Write $t_b=b/K$, $\sigma=dw$, $\Lambda=nb^3/K$, and $L_*=16$. Then
\begin{equation}\label{eq:Sorders}
t_b\asymp d^{-4},\quad A\asymp d^{-6},\quad
\Lambda\asymp n/d^8,\quad \Lambda/d^2\to\infty,\quad
w\asymp\sqrt n/d^2\to\infty,\quad
\frac{w}{nA}=\frac2{\sqrt\Lambda}\to0.
\end{equation}
Also $d^3/n\to0$, $nt_b/d^2\to\infty$, and $d^2=o(n^{1/3})$.

\subsection{Exact double zero and deterministic margins}
Set $h_\alpha(z)=\alpha+(1-\alpha)F(z)-z$. The maximizer in
Theorem~\ref{thm:D} satisfies the exact identities
\begin{equation}\label{eq:doublezero}
h_{\adet}(\widehat z)=h'_{\adet}(\widehat z)=0,
\qquad h_{\adet}(z)=B(z)[\adet-\mathcal R(z)]\ge0.
\end{equation}
For the derivative, $1-\adet=(1-\widehat z)/B(\widehat z)$ and
$g(\widehat z)=0$ in~\eqref{eq:g}. Let $z_L=L_*t_b$, $u_L=-\log(1-z_L)$.
On this interval $F''\ge K/2$, $\adet\to0$, and two integrations from the
double zero yield
\begin{equation}\label{eq:Scurvature}
D_{n\adet}(u)\ge kndK(z-\widehat z)^2\quad(0\le z\le z_L).
\end{equation}
Here real seed counts occur only as deterministic references.
For $a=O(nA)$, Taylor expansion and seed depletion give uniformly
\begin{equation}\label{eq:Sgapdrift}
h_{a/n}(z)=a/n-A+\tfrac K2(z-t_b)^2
+O(Ab/d+Ab/K)\quad(0\le z\le z_L).
\end{equation}
Thus a fixed relative negative gap gives a negative margin $-kA$ at $z=t_b$;
a fixed relative positive gap gives a positive early margin $kA$.

The deterministic margin proof in Lemma~\ref{lem:rangefinal} applies
from $z_L=16b/K$: $F-z\ge Kz^2/8$ until $\eta/d$, and has a fixed
positive lower bound thereafter through $z=1/2$. Consequently
\begin{equation}\label{eq:Slatemargins}
D_a(u)\ge kndKz^2\ (z_L\le z\le\eta/d),\qquad
D_a(u)\ge knd\ (\eta/d\le z\le1/2).
\end{equation}
These bounds hold for negative as well as positive window coordinates $x$.

\subsection{Specialization of the early Gaussian estimate}
For $a=\lfloor n\adet+xw\rfloor$ with fixed $x$,~\eqref{eq:Sorders}
gives $0<a<n$, $a=O(nA)$ and actual coordinate
$(a-n\adet)/w=x+O(1/w)$. Apply Lemma~\ref{lem:Gloc} and its
budget~\eqref{eq:qGearly}. Here
\[
r_Da/(nt_b)=O(r_Db)=O(d^{-3/2}),\qquad
\mathcal E_G=O(d^{-1}),\qquad \zeta_G=O(d^{-1}),
\]
while $\Lambda\asymp n/d^8\to\infty$ and $w\to\infty$.
Thus, for the physical early-survival event $E=\{\tau>u_L\}$,
$\Pp_{\Conf}(E)=\Phi(x)+o(1)$.
Use~\eqref{eq:qEventTransfer} with $m=\lceil4L_*ndt_b\rceil$:
$u_L\le2L_*t_b$, overflow is $e^{-\Omega(ndt_b)}$, and
$md/n+d^3/n=O(b+d/n+d^3/n)=o(1)$. Therefore
\begin{equation}\label{eq:Searly}
\Pp_\Simple(E)\longrightarrow\Phi(x).
\end{equation}
The late and terminal estimates below supply the additional conclusions
of S and the remaining terms in its quantitative budget.

\subsection{Late escape: budgets, exponents, and conditioning}
For fixed-gap comparisons use budget $e_0=\rho ndA$ on $[0,z_L]$.
The dyadic middle intervals start at $z_L$ and have endpoints
$h$ and $\min(2h,\eta/d)$.
Use budget $e_h=\rho ndKh^2$ on each, and $e_p=\rho nd$ on the final interval.
Here $\rho>0$ is fixed and sufficiently small relative to the deterministic
margins. The corresponding $u$ interval lengths are $O(h)$, because
$\dd u/\dd z=(1-z)^{-1}\le2$ throughout. All budgets tend to infinity,
so the reserved-end correction one is absorbed.

For $a=O(nA)$, the early variance bounds are
$V=O(nb)$, $a\Delta^2=O(nb^2)$, $Mr=O(nb/d)$ and
$\Delta e_0=O(nb^2)$. Lemma~\ref{lem:interval} therefore gives vertex,
seed-mean, and free-end exponents at least, respectively,
$k\Lambda,k\Lambda/b,kd\Lambda$. On a dyadic middle interval the respective
lower bounds are
\begin{equation}\label{eq:Sexponents}
k\frac{nd^4h^3}{1+d^2h},\qquad
k\frac{nd^6h^4}{b^2},\qquad
k\frac{nd^5h^3}{1+d^2h}.
\end{equation}
These are~\eqref{eq:Bexponents} with $K\asymp d^2$,
$a=O(nb^2/K)$ and equivalent $z,u$ intervals. Each increases with $h$;
at $h=z_L\asymp b/d^2$ it is at least the early lower bound.
The final interval has exponents at least $kn$.

The early grid has $O(1+d/b)$ points. Each middle grid has
$O(1+1/(dh))\le O(1+d/b)$ points, and there are
$O(1+\log(d/b))$ such intervals. The final grid has $O(1+d)$ points.
Thus one common failure bound is
\begin{equation}\label{eq:Slocalfailure}
\varepsilon_{\mathrm{loc}}=C(1+d/b)(1+\log(d/b))e^{-k\Lambda}.
\end{equation}
No independence between intervals is used. For window sizes use only middle
and final intervals and the margins~\eqref{eq:Slatemargins}; a stop there
would contradict the positive queue at $\tau$. Hence
\[
\Pp_\Conf(E,\tau\le U)\le\varepsilon_{\mathrm{loc}},\qquad
\Pp_\Simple(E,\tau\le U)\le e^{C_sd^2}\varepsilon_{\mathrm{loc}}=o(1).
\]
This division uses an exponentially small intersection error. It is separate
from the total-variation transfer of~\eqref{eq:Searly}.

\subsection{Terminal dichotomy and fixed gaps}
At $u_L$, the independent all-vertex indicators
$\ind_{\{X_i(u_L)\ge R_i\}}$ have mean sum at most
\[
\sum_i[\pi d_i z_L+\tbinom{d_i}{2}z_L^2]\le Cnt_b.
\]
Bernstein bounds their sum by $C'nt_b$ except with probability $e^{-knt_b}$.
Every nonseed activation completed before a stop $\tau\le u_L$ belongs to
that sum. Adding $a=O(nt_bb)$ gives
\begin{equation}\label{eq:Smallterminal}
\Pp_\Simple(E^c,T>C'nt_b)\le e^{C_sd^2-knt_b}=o(1).
\end{equation}
At $U=\log2$ take $\varepsilon_T=n^{-1/3}$ in the general terminal
calculation in Lemma~\ref{lem:rangefinal}, so $E_\ell=n^{2/3}d$.
Its conditional Hoeffding and sampled-seed bounds give
$\Pp(|H(U)-\bar H(U)|>2E_\ell)\le4e^{-kn^{1/3}}$ (apply the same
two-sided inequalities). The exact mean is
$\bar H(U)=(1-a/n)\sum_i d_i2^{-d_i}[1+(1-\pi)(d_i-1)]$.
Consequently on $\tau>U$, outside that bad event,
\begin{equation}\label{eq:residual}
\frac{\bar H(U)+2E_\ell}{\delta-1}
\le B_+n[d2^{-d+C\sqrt d}+n^{-1/3}]=o(n).
\end{equation}
Simplicity division preserves this terminal estimate since
$d^2=o(n^{1/3})$. It never assumes latent independence conditional on
survival or simplicity. Equations~\eqref{eq:Smallterminal}--\eqref{eq:residual}
and late escape give
$\Pp_\Simple(T\ge n/2)-\Pp_\Simple(E)\to0$.
Together with~\eqref{eq:Searly} this proves the Gaussian assertion.

For the negative fixed gap take
$a_- =\lfloor(1-\gamma)nA\rfloor$.
Equation~\eqref{eq:Sgapdrift} is at most $-kA$ at $z=t_b$.
Early concentration forces stopping before that point: otherwise the real
queue there would be less than $-1$. The same latent activation count then
gives $T=O(nt_b)=O(n/d^4)$ after the small error is divided by simplicity.
For the positive gap take $a_+=\lceil(1+\gamma/2)nA\rceil$.
There is a positive early margin and all later margins; concentration excludes
$\tau\le U$, giving~\eqref{eq:residual}. The common failure envelope is
\[
e^{C_sd^2}[\varepsilon_{\mathrm{loc}}+e^{-knt_b}+4e^{-kn^{1/3}}]=o(1)
\]
by~\eqref{eq:Sorders}. Finally an independent uniform vertex permutation
couples all uniform seed sets as nested prefixes on the same graph and
threshold assignment. Activation is monotone in seeds. Every prescribed
$a\le(1-\gamma)nA$ is below $a_-$, and every
$a\ge(1+\gamma)nA$ is above $a_+$ eventually; $nA\to\infty$ absorbs rounding.
The bounds transfer by monotonicity, including the trivial empty/full cases.
This completes Theorem~\ref{thm:S}.

\subsection{Quantitative terminal rate}
In the polynomial S range the separately proved late/terminal envelope is
\begin{equation}\label{eq:qGlateS}
\mathcal L_G=e^{C_sd^2}
[C(1+d/b)(1+\log(d/b))e^{-k\Lambda}
 +e^{-knt_b}+4e^{-kn^{1/3}}].
\end{equation}
Here $\mathcal E_G=O(d^{-1})$, $\zeta_G=O(d^{-1})$ and
$w^{-1/4}=O(d^{1/2}n^{-1/8})$. Therefore uniformly for $|x|\le X$,
\begin{align}
|\Pp(T\ge n/2)-\Phi(x)|\le{}&
C_X[d^{-1/3}+(d^8/n)^{1/13}+d^{1/2}n^{-1/8}]\nonumber\\
&+O_E(b+d^3/n)+\mathcal L_G+2\mathcal O.
\label{eq:qSrate}
\end{align}
The constants may depend on the fixed limiting $\beta$; $\mathcal O$
is~\eqref{eq:qOverflow} at the stated cap/horizon. Rounding a target
coordinate costs only $C/w$. The exponent $1/13$ comes from the localization choice in
Section~\ref{subsec:Gloc}. The budget includes the final macro event
in S's polynomial range.

\section{Nonempty ranges and the limitation of this proof route}
\label{sec:examples}
\subsection{Fixed-gap and regular examples}
Suppose $c+1-v\ge\kappa>0$. Identity~\eqref{eq:bexact}
then gives $b\asymp d^{-1}$ and $\Lambda\asymp n/d^5$. Our barrier
expansion gives
\[
\adet=\frac{b^2}{2K}(1+O(b/d)),\qquad
\frac{n\adet}{w}\sim\frac{\sqrt\Lambda}{2}.
\]
If $n/d^5\to\infty$, fixed relative offsets from this threshold
have B coordinates tending to minus or plus infinity. Consequently,
Theorem~\ref{thm:B} gives macro-event probabilities tending to zero or
one, respectively.

For an exactly regular sequence put
$\beta_R=(1-\pi)(d-1)-1=d-3+b$. Then
\[
B(z)=(1-z)^{d-2}(1+\beta_Rz),\qquad
\mathcal R(z)=1-\frac{(1-z)^{3-d}}{1+\beta_Rz}.
\]
The logarithmic derivative of the fraction on the right is
\[
\frac{d-3}{1-z}-\frac{\beta_R}{1+\beta_Rz}
=\frac{-b+\beta_R(d-2)z}{(1-z)(1+\beta_Rz)}.
\]
Here $b>0$ and $\beta_R>0$ eventually, so the unique maximum is at
$z_R=b/[\beta_R(d-2)]$. Consequently
\begin{equation}\label{eq:regularcenter}
n\adet=n\left[1-\frac{(1-z_R)^{3-d}}{1+\beta_Rz_R}\right]=C_R,
\qquad \frac{w}{w_R}\longrightarrow\sqrt{c+1},
\quad w_R=\sqrt{n/d^3}.
\end{equation}
For seeds $a=\lfloor C_R+xw_R\rfloor$ with fixed $x\in\mathbb R$,
the B coordinate tends to $x/\sqrt{c+1}$, since $w\to\infty$.
Thus Theorem~\ref{thm:B} yields
\[
\Pp_{\Simple}(T_a\ge n/2)\longrightarrow\Phi(x/\sqrt{c+1})
\qquad\text{when }n/d^5\to\infty.
\]

\subsection{Examples and the horizon boundary}
For integer $d=t\to\infty$ and even $n=2\lceil t^p/2\rceil$, $p>4$,
the exactly regular sequence is graphical by the cyclic construction in
Lemma~\ref{lem:graphical}. It has
\[
b=(c+1)/d-c/d^2,\quad K\sim d^2,\quad
\lambda\sim(c+1)d^{(p-5)/3},\quad
\Lambda\sim(c+1)^3d^{p-5}.
\]
For $4<p<5$ it realizes the Brownian law with $\lambda\to0$; for $p=5$
it realizes finite positive $\lambda$; for $p>5$ it realizes the Gaussian
layer. More generally $|b|=O(d^{-1})$ implies
$|\lambda|=O((n/d^5)^{1/3})$. Thus every envelope array with
$d\asymp n^\kappa$, $1/5<\kappa<1/4$, has $\lambda\to0$:
X/C are substantive, whereas B is empty there.

Under the strict variance-capacity margin $C^2>c+1$,
Proposition~\ref{prop:realize} attains actual
$b\sim\beta/d^2$, $\beta>0$, with arithmetic error $O((nd^2)^{-1})$.
Here $\Lambda\sim\beta^3n/d^8$. Consequently B proves the exactly
centered terminal probability window under $n/d^8\to\infty$, including
$n\asymp d^p$, $8<p\le10$, outside the S terminal-size regime.
No $O(n/d^4)$ early residue or $O(n^{2/3})$ late residue is transferred
to that whole new range.

Arbitrarily slow degree divergence also satisfies X/C. For example
$d=\lfloor\log\log\log n\rfloor$ with even $n$ gives
$J\asymp\log n$ and $\Psi/J\asymp\sqrt{n/(d^2\log n)}\to\infty$.
The covariance errors still vanish because $d\to\infty$; fixed degree
is excluded. Near the upper range, $n=d^4L$ with $L\to\infty$ and
$\log(n+2)=o(d^2)$ gives
\[
\theta\asymp L^{-1/6},\quad h_0\asymp d^{-2}L^{-1/6},\quad
q_0\asymp d^{2/3}L^{1/6},\quad \Psi\asymp d^2\sqrt L.
\]
Hence the conditioning exponent beats $d^2$ even for arbitrarily slow
$L\to\infty$.

\begin{proposition}[A barrier of the stated proof route]
\label{prop:routebarrier}
If $n/d^4=O(1)$ on a subsequence, the two horizon inequalities in
Lemma~\ref{lem:rangehorizon} have no solution. Within the enumeration
domain $d^2/n\to0$, no horizon with vanishing early TV cost
$d^2h=o(1)$ can make the present unconditional vertex Bernstein
bound, followed by full simplicity division, tend to zero. This is a
method obstruction, not a counterexample to the cascade law.
\end{proposition}
\begin{proof}
Put $t=d^2h\to0$. The exponent derived in~\eqref{eq:Bexponents}
has scale
$I_V(h)\asymp(n/d^2)t^3/(1+t)$. Hence
$I_V(h)/d^2\asymp(n/d^4)t^3/(1+t)\to0$.
For the original near-regular sequence,
$\nu=d-1+\operatorname{Var}(D)/d=d+O(1)$, so Input E gives
$-\log s=d^2/4+O(d+d^3/n)=d^2/4+O(d)$ within its domain.
The displayed exponential upper-bound route therefore cannot beat
the simplicity cost, even before the grid prefactor is inserted.
Equivalently the two horizon inequalities in Lemma~\ref{lem:rangehorizon}
have no solution. This does not exclude conditional concentration or
another proof at $d\asymp n^{1/4}$, and claims no optimality of $1/4$.
\end{proof}

The successful horizon does not imply a Brownian approximation through
$q_0$. For polynomial degrees with $J\asymp d^2$, the cubic response
error normalized by $G$ at that time has order
$du_0q_0^3\asymp(n/d^2)^{1/6}$, which diverges. Its error relative to
the late quadratic margin is instead $O(dh_0)=o(1)$.
The proofs use a compact Brownian limit followed by exact-response
positive margins, precisely to avoid that invalid inference.

\section{Theorem, assumption, and lemma matrix}\label{sec:ledger}
Table~\ref{tab:dependencies} separates the main qualitative regimes from
the narrower quantitative and relative assertions. All graph rows retain
the prescribed simple ensemble, fixed independent uniform seeds and iid
thresholds of~\eqref{eq:envelope}. E is the only external graph-enumeration
input; all other dependencies listed in the table are proved here. The main proof
order is X, quantitative comparison tools, B/C, and then S; the S-specific
terminal rate is collected with its proof in Section~\ref{sec:S}.

\begingroup
\small
\setlength{\tabcolsep}{4pt}
\begin{longtable}{@{}>{\raggedright\arraybackslash}p{.17\linewidth}>{\raggedright\arraybackslash}p{.32\linewidth}>{\raggedright\arraybackslash}p{.45\linewidth}@{}}
\caption{Results, assumptions, and proof dependencies.}
\label{tab:dependencies}\\
\toprule
Result & Additional assumptions & Proved dependencies / scope \\
\midrule
\endfirsthead
\toprule Result & Additional assumptions & Proved dependencies / scope \\
\midrule\endhead
\bottomrule\endfoot
Barrier D & Envelope only; actual $b$, with no positive lower bound &
Finite derivatives~\eqref{eq:curvfinite}, quotient identity~\eqref{eq:g},
Taylor remainders. Exact zero center if $b\le0$. \\
Graphical tuning & Integer mean, even $n$, $C^2>c+1$, interior variance &
Lemma~\ref{lem:graphical}, Proposition~\ref{prop:realize}; nearest
variance lattice, not guaranteed exact arithmetic zero. \\
Adaptive transfer & $m=o(n/d)$, $d^3/n=o(1)$ &
E, residual subtraction~\eqref{eq:qdiff}, residual-edge switching;
Lemma~\ref{lem:TV}. Applies to capped physical transcripts. \\
Horizon / terminal passage & $d^4/n\to0$; X or B comparison sizes &
Lemmas~\ref{lem:rangehorizon}, \ref{lem:rangefinal}, nested estimates,
interval concentration, exact fixed-seed sampling, E. \\
B & $b>0$, $\Lambda\to\infty$ &
D; Lemma~\ref{lem:Gloc} with scalar comparison
and budget~\eqref{eq:qGfinite}; summable bridge~\eqref{eq:Bsum}; common terminal
passage; monotone seed tails. Implies its degree bound. \\
X & $d^4/n\to0$, finite $\lambda$ and positive finite $a/a_0$ &
Lemmas~\ref{lem:CLT}, \ref{lem:compact}; latent first-hit covariance
and hazard; stopped-event sandwich; dyadic physical bridge;
common terminal passage. \\
C & $d^4/n\to0$, $\lambda\ge-L_0$ &
B, X, joint continuity of $p$, Brownian consistency
\eqref{eq:BfunctionUniform}, monotone tails and subsequences.
Absolute error. \\
Exact-center counterexample & Logarithmic degrees, tuned
$b\sim\beta/d^2>0$, $C^2>c+1$ &
Graphical tuning, D cubic correction, B, uniform Brownian consistency.
Actual error $1/2$ for the naive formula. \\
S terminal sizes / rate & $b\sim\beta/d^2$, $n/d^{10}\to\infty$ &
Section~\ref{sec:S}; separate local exponential and terminal budgets;
\eqref{eq:qSrate}. \\
Bounded-layer rate & $d\asymp\log^2n$, bounded $\lambda$ and compact
positive seed coordinate &
Finite smoothing~\eqref{eq:qCompactBound}, dyadic tails, restricted late
budget~\eqref{eq:qXlate}; rate~\eqref{eq:qXrate}. \\
Brownian boundary & $x\downarrow0$, $\lambda$ in a fixed compact &
Reflection and deterministic Gaussian shift; killed radial identity;
relative horizon bound~\eqref{eq:seedrelative}. \\
Mesoscopic relative law & $d\asymp\log^2n$, bounded $\lambda$,
$a_0/\log d\ll a\ll a_0$ &
Boundary coefficient; explicit shrinking-coordinate check of
\eqref{eq:qCompactBound}; estimate~\eqref{eq:seedabsolute}. \\
Fixed-seed entrance & Same logarithmic / bounded-$\lambda$ regime,
fixed $a\ge1$, $H=\lceil d^2\rceil$ &
Distinct root/offspring laws, truncated optional stopping, fresh-vertex
coupling, E-based transcript transfer. Not a terminal law. \\
Proof-route barrier & $n/d^4=O(1)$ &
Unconditional vertex exponent versus genuine simplicity cost;
Proposition~\ref{prop:routebarrier}. Not optimality or a graph counterexample. \\
\end{longtable}
\endgroup

\appendix

\section{Brownian boundary and mesoscopic seeds}\label{app:seed}
All graph assertions in this appendix assume $d\asymp\log^2n$ with
fixed comparison constants and $|\lambda|\le L$. Brownian assertions
have no graph-degree assumptions.

\paragraph{Relation to classical parabolic maxima.}
For $x>0$, Brownian symmetry rewrites the survival function as
\[
p(x,\lambda)=\Pp\!\left(\sup_{r\ge0}
\{\mathcal B(r)+\lambda r-r^2/2\}<x\right).
\]
For $\lambda=0$ this is $F_{1/2}(x)$ in
\citet[Section~2, Lemma~2.1]{groeneboom2010}.
For general $\lambda\in\mathbb R$, let $Q_{s,y}$ denote the law in
\citet[Theorem~1.1]{glt2013} of Brownian motion minus $t^2$, started
at position $y<0$ at time $s\in\mathbb R$, and let $\tau_0$ be its
first hitting time of zero. Set $s=-2^{-2/3}\lambda$ and
$y=-2^{-1/3}x$. Brownian scaling $r=2^{2/3}v$ transforms
$-x+\mathcal B(r)+\lambda r-r^2/2$, divided by $2^{1/3}$, into
$y+\widetilde{\mathcal B}(v)-2sv-v^2$. Consequently
\[
p(x,\lambda)=Q_{s,y}(\tau_0=\infty).
\]
If $k(s,y)=\partial_yQ_{s,y}(\tau_0<\infty)$, their
Theorem~3.1 gives the boundary limit
$k(s,0)=\lim_{y\uparrow0}k(s,y)$, and hence
\[
\kappa(\lambda)=2^{-1/3}k(-2^{-2/3}\lambda,0).
\]
This is the derivative of a one-sided maximum distribution, not the
Chernoff distribution of the location of a two-sided maximum.
The proof below establishes the boundary law and its radial
representation directly; no Airy identity is used as a proof input.

\subsection{A self-contained boundary coefficient}
Let $p_t(x,\lambda)$ be survival of
$x+\mathcal B(s)-\lambda s+s^2/2$ through $[0,t]$, and put
\[
J_t(\lambda)=\int_0^t(s-\lambda)^2\,\dd s
=t^3/3-\lambda t^2+\lambda^2t.
\]
Let $W$ consist of three independent standard Brownian motions, and
$R_s=|W_s|$. Set
\begin{equation}\label{eq:seedkappat}
\kappa_t(\lambda)=e^{-J_t(\lambda)/2}
\E\left[\frac{\exp\{(t-\lambda)R_t-\int_0^tR_s\,\dd s\}}{R_t}\right].
\end{equation}

\begin{theorem}[Boundary coefficient]\label{thm:boundary}
The limit $\kappa(\lambda)=\lim_{t\to\infty}\kappa_t(\lambda)$ is
finite, strictly positive and continuous. Uniformly for $\lambda$ in
each fixed compact interval,
$p(x,\lambda)=\kappa(\lambda)x+o(x)$ as $x\downarrow0$.
For $t>\max(\lambda,0)$,
\begin{equation}\label{eq:seedderivativetail}
0\le\kappa_t(\lambda)-\kappa(\lambda)
\le e^{-J_t/2}\frac2t
\int_0^\infty y\phi_t(y)e^{-(t-\lambda)y}\,\dd y,
\end{equation}
where $\phi_t(y)=(2\pi t)^{-1/2}e^{-y^2/(2t)}$.
For any fixed $t>0$ a nonzero-height formula is
\begin{equation}\label{eq:seedkappa}
\kappa(\lambda)=e^{-J_t/2}
\E\left[\frac{e^{(t-\lambda)R_t-\int_0^tR_s\,\dd s}}{R_t}
p(R_t,\lambda-t)\right].
\end{equation}
\end{theorem}
\begin{proof}
For $Y_s=x+\mathcal B(s)$ killed at zero, reflection
(Lemma~\ref{lem:Brownianmin}) gives its transition density
\begin{equation}\label{eq:seedkernel}
k_t(x,y)=\phi_t(y-x)-\phi_t(y+x),\qquad x,y>0.
\end{equation}
For $R_s^x=|(x,0,0)+W_s|$, integrating the three-dimensional Gaussian
density over a radius-$y$ sphere gives
$k_t^{(3)}(x,y)=(y/x)k_t(x,y)$. Indeed the angular factor is
$2\pi\int_{-1}^1e^{xyv/t}\,\dd v
=4\pi t\sinh(xy/t)/(xy)$, which, multiplied by
$y^2(2\pi t)^{-3/2}e^{-(x^2+y^2)/(2t)}$, gives that formula.
Rotational invariance gives the radius its Markov transition law.
Multiplying the density identity over a finite time grid telescopes
the factors $y/x$, proving for nonnegative cylinder functionals $F$
\begin{equation}\label{eq:seedradial}
\E_x[\ind_{\{\tau_0>t\}}F(Y)]
=x\E[F(R^x)/R_t^x].
\end{equation}
A monotone-class argument extends it to nonnegative path functionals.
Continuous paths are determined by rational times; the killed densities
encode survival. This is a direct Gaussian calculation, not an imported
meander or diffusion entrance theorem.

The deterministic Gaussian change of measure proved in
Lemma~\ref{lem:Brownianmin}, with $f(s)=s-\lambda$, has density
$\exp\{\int_0^t f\,\dd\mathcal B-J_t/2\}$. Integration by parts is exact:
\[
\int_0^t(s-\lambda)\,\dd\mathcal B(s)
=(t-\lambda)Y_t-\int_0^tY_s\,\dd s+\lambda x.
\]
Consequently
\begin{equation}\label{eq:seedfiniteheight}
\frac{p_t(x,\lambda)}x=e^{\lambda x-J_t/2}
\E\left[\frac{e^{(t-\lambda)R_t^x-\int_0^tR_s^x\,\dd s}}{R_t^x}\right].
\end{equation}
Under a common construction $\sup_{s\le t}|R_s^x-R_s|\le x$.
For every $q<3$,
$\sup_{0\le x\le1}\E(R_t^x)^{-q}<\infty$: inside the unit ball
bound the shifted Gaussian density by $(2\pi t)^{-3/2}$ and integrate
$r^{2-q}$; outside, the inverse power is bounded. Positive exponential
moments of $R_t^x$ are uniformly finite since $R_t^x\le1+|W_t|$.
H\"older, with an inverse-power exponent below three and an exponent
slightly above one for the whole integrand, proves uniform integrability.
The negative area term can be dropped for this bound. Passage to $x=0$
in~\eqref{eq:seedfiniteheight} proves~\eqref{eq:seedkappat}, uniformly
on compact $\lambda$ intervals. The same domination proves continuity
of $\kappa_t$.

Condition at time $t>\lambda$. At height $y>0$ the future drift is
$(t-\lambda)v+v^2/2$, so the future hit probability is at most
$e^{-2(t-\lambda)y}$ by the positive linear-drift martingale in
Lemma~\ref{lem:Brownianother}. Retaining the change-of-measure endpoint
factor, and dropping the nonpositive area exponent, gives
\begin{equation}\label{eq:seedrelativet0}
0\le p_t(x,\lambda)-p(x,\lambda)
\le e^{\lambda x-J_t/2}
\int_0^\infty k_t(x,y)e^{-(t-\lambda)y}\,\dd y.
\end{equation}
For $0<x\le1$,
\[
\frac{k_t(x,y)}x
=\phi_t(y-x)\frac{1-e^{-2xy/t}}x
\le\frac{2y}{t}\phi_t(y)e^{y/t}.
\]
If $|\lambda|\le L$ and $t\ge\max\{6(L+1),2\}$, then
$t-\lambda-1/t\ge t/2$ and $J_t\ge t^3/6$. Therefore
\begin{equation}\label{eq:seedrelative}
\sup_{\substack{|\lambda|\le L\\0<x\le1}}
\frac{p_t(x,\lambda)-p(x,\lambda)}x
\le C_Lt^{-7/2}e^{-t^3/12}.
\end{equation}
The power is $t^{-3/2}\int_0^\infty ye^{-ty/2}\,\dd y=O(t^{-7/2})$.
The factor $x$ is essential: this is a relative horizon estimate.

Since $p_t$ decreases with $t$, $\kappa_t$ decreases to a nonnegative
finite limit. First taking a large fixed $t$ and then small $x$ in
\eqref{eq:seedfiniteheight} and~\eqref{eq:seedrelative} proves the
compact-uniform boundary expansion and uniform convergence of
$\kappa_t$. The derivative of~\eqref{eq:seedkernel} at $x=0$ is
$2y\phi_t(y)/t$; passage to that derivative in
\eqref{eq:seedrelativet0} proves~\eqref{eq:seedderivativetail}.

The deterministic-time Markov property replaces the integrand in
\eqref{eq:seedfiniteheight} by the same integrand times
$p(R_t^x,\lambda-t)$ for infinite-horizon survival. The function $p$
is jointly continuous for positive height by the finite minimum no-atom
law and the uniform tail~\eqref{eq:qBrownianTail}. The same domination
therefore proves~\eqref{eq:seedkappa}. The positivity $p(y,\nu)>0$ for every $y>0$ and finite $\nu$
was proved in Section~\ref{sec:X} by a Brownian tube followed by
linear-drift survival. Thus the integrand in
\eqref{eq:seedkappa} is strictly positive almost surely and integrable.
This proves $\kappa>0$. Uniform convergence of the continuous
$\kappa_t$ proves continuity and hence
\begin{equation}\label{eq:seedcompactpositive}
0<\inf_{|\lambda|\le L}\kappa(\lambda)
\le\sup_{|\lambda|\le L}\kappa(\lambda)<\infty.
\end{equation}
\end{proof}

\subsection{Distinct roots and the discrete entrance law}
Let $D^*$ have the size-biased empirical degree law
$\Pp(D^*=k)=k\Pp(D=k)/d$. A uniform seed has vulnerable-child law
\[
\xi^{\rm root}\mid D=k\sim\operatorname{Bin}(k,\pi),\qquad
\E\xi^{\rm root}=1-c/d,
\]
whereas a reached nonroot has
\begin{equation}\label{eq:seedoffspring}
\xi\mid D^*=k\sim\operatorname{Bin}(k-1,\pi),\qquad \E\xi=1-b.
\end{equation}
The arrival edge is used once. Repeating the root deficit as the
descendant drift would give an incorrect branching law.
The mixture has bounded exponential moments near zero, since
$\pi\Delta=O(1)$, and $v_2:=\E(\xi-1)^2\to1$.
Indeed its conditional variance is $(1-\pi)\E\xi\to1$, the
variance of its conditional mean is
$\pi^2\Var(D^*-1)=O(d^{-1})$, and the squared mean deficit is $b^2=o(1)$.

\begin{lemma}[Branching entrance]\label{lem:seedgw}
Let $H\to\infty$, $\Delta/H\to0$ and $|b|H\to0$.
For $S_j=z+\sum_{i=1}^j(\xi_i-1)$ with independent offspring
\eqref{eq:seedoffspring}, let $\sigma_H$ be its first hit of zero or
$[H,\infty)$. Uniformly over integers $z$ with $1\le z<H$,
\begin{equation}\label{eq:seedgw}
\E_z\sigma_H\le Cz(H+\Delta),\qquad
\Pp_z(S_{\sigma_H}\ge H)=\frac zH
[1+O(|b|H+\Delta/H)].
\end{equation}
For fixed $a$ roots processed first, if $a\Delta<H$, the hit
probability is $a(1-c/d)/H$ times the same relative error.
\end{lemma}
\begin{proof}
Before exit, the conditional mean increment of $S_j^2$ is
$v_2-2bS_j$, bounded below by a fixed positive constant.
At $t_N=\sigma_H\wedge N$, $0\le S_{t_N}\le H+\Delta$. Thus
\[
k\E t_N\le\E S_{t_N}^2-z^2
\le(H+\Delta)\E S_{t_N},\qquad
\E S_{t_N}=z-b\E t_N.
\]
For $b\ge0$ the mean is at most $z$; for $b<0$ absorb
$|b|(H+\Delta)\E t_N$. Let $N\to\infty$ to get finite mean exit.
Bounded stopped positions and this truncation justify optional stopping,
giving $\E S_{\sigma_H}=z[1+O(|b|H)]$. Extinction is at exactly zero,
since the queue is integer-valued and its downward step is $-1$;
a hit height lies in $[H,H+\Delta]$.
Division gives~\eqref{eq:seedgw}. Given the total root child count $Z$,
descendant offspring are independent of it; $Z\le a\Delta<H$ and
$\E Z=a(1-c/d)$. Averaging proves the root assertion.

The uniform mean bound is at most $CH^2$ from any integer $z$ with
$1\le z<H$.
In blocks of $\lceil2CH^2\rceil$ steps, conditional Markov inequality
and independent future offspring give survival chance at most $1/2$.
Iteration proves
\begin{equation}\label{eq:seedgwtail}
\sup_{\substack{z\in\mathbb Z\\1\le z<H}}
\Pp_z(\sigma_H>N)\le C'e^{-k'N/H^2}.
\end{equation}
\end{proof}

\begin{theorem}[Fixed-seed simple-graph entrance]\label{thm:entrance}
In the logarithmic-degree, bounded-$\lambda$ graph regime, fix
$a\ge1$. Sample the seeds in uniform random order without replacement
and process them first in that order. Process active nonseeds in
first-in, first-out (FIFO) order of discovery, revealing each processed
vertex's half-edges in a fixed order. Put $H=\lceil d^2\rceil$ and
$N=\lceil H^2d\rceil$. Let $E_H$ be the event that after roots are
processed, the nonseed processing queue reaches at least $H$ before
exhaustion, within $N$ nonroot steps. Then
\[
\Pp_{\Simple}(E_H)=\frac aH(1+o(1)),\qquad
\Pp_{\Simple}(T_a\ge n/2,E_H^c)=o(a/a_0).
\]
These are entrance assertions; they do not prove the fixed-$a$ terminal law.
\end{theorem}
\begin{proof}
Cap pair revelations at $m=\Delta(a+N)=O(d^6)$. In the configuration
exploration abort comparison if a partner belongs to any already seen
vertex, including roots and inactive vertices. Before reveal $j$ there
are at most $a+j$ such vertices and at most $(a+j)\Delta$ ends there.
The partner pool has size at least $nd-2j-1$, so the repeat chance is
$C(a+j)/n$. A fresh partner's degree distribution is proportional to
$k(n_k-\mathrm{seen}_k)$; its TV distance from the original size-biased
law is also at most $C(a+j)/n$. Couple successively to independent
draws of that law. Fresh threshold labels are independent Bernoulli
variables and independent of their degrees. Root identities, sampled
without replacement, couple to iid uniform identities at cost $O(a^2/n)$.
The combined error is
\begin{equation}\label{eq:seedtreecoupling}
\delta_{\rm tree}\le C(a+m)^2/n=O(d^{12}/n).
\end{equation}
On a successful coupling every partner is fresh, so threshold-two
vertices receive only one mark and do not activate. A vulnerable nonroot
uses its parent edge, leaving exactly the offspring law
\eqref{eq:seedoffspring}; the independent child degree draws are
independent of its parent's child count. In the coupled branching
construction, a child's degree can be left unread until it is processed:
its vulnerable label and its position in the FIFO queue are independent
of that degree. Thus successive nonroot offspring have the independent
mixture law~\eqref{eq:seedoffspring}. The root law remains distinct,
and the physical entrance agrees with the branching queue event.

The actual exploration is an adaptive capped transcript, so
Lemma~\ref{lem:TV} adds error
$O_E(md/n+d^3/n)=O_E(d^7/n+d^3/n)$. At logarithmic degrees these
errors are $o(H^{-1})$ and $o(a_0^{-1})$, with all residual E domains
satisfied. Here $|b|H\le LH/a_0\to0$, $\Delta/H=O(d^{-1})$ and
$a\Delta<H$. Lemma~\ref{lem:seedgw} gives the untruncated hit probability
$a/H(1+o(1))$, and~\eqref{eq:seedgwtail} makes truncation cost
$Ce^{-kd}$. This proves the entrance law.

If $n/2$ vertices activate without $E_H$, the processing queue cannot
have exhausted among its first $a+N=o(n)$ processed vertices. It
therefore survives $N$ nonroot steps without hitting $H$. The successful
coupling makes this a branching exit-time-tail event, giving
\[
\Pp_{\Simple}(T_a\ge n/2,E_H^c)
\le Ce^{-kd}+O(d^{12}/n)+O_E(d^7/n+d^3/n)=o(a_0^{-1}).
\]
Indeed $\log a_0\sim\log n/3$, $d\asymp\log^2n$, and every
polynomial in $d$ divided by $n$ is negligible compared with $a_0^{-1}$.
\end{proof}

\subsection{Shrinking-coordinate comparison and the mesoscopic law}
The shrinking-coordinate estimate~\eqref{eq:seedabsolute} was proved
above directly from~\eqref{eq:qCompactBound}, including its strict
height condition, all four error budgets, and the physical late passage.
Thus no substitution of vanishing $x$ into a fixed-positive-$x$ limit is needed.

\begin{proof}[Proof of Theorem~\ref{thm:meso}]
Its assumptions give $x=a/a_0\downarrow0$ and $x\log d\to\infty$,
hence $x\gg d^{-1/8}$. Divide~\eqref{eq:seedabsolute} by $x$ to make
the error vanish. The compact-uniform expansion in
Theorem~\ref{thm:boundary} gives
$p(x,\lambda)=\kappa(\lambda)x+o(x)$, and
\eqref{eq:seedcompactpositive} permits division by that proposed probability.
Continuity permits replacement by the limiting coefficient if
$\lambda\to\ell$. For example
$a=\lfloor a_0/(\log d)^{1/2}\rfloor$ is a nonempty mesoscopic scale
on the graphical arrays constructed above. No inference about $a=1$
is made.
\end{proof}

\subsection{What is missing for fixed seeds}
For fixed $a\ge1$, the conjectured terminal law is
$\Pp(T_a\ge n/2)\sim\kappa(\lambda)a/a_0$. The entrance theorem
reduces it to an entrance-conditioned continuation estimate: one would
need, with the actual entrance history distribution,
\begin{equation}\label{eq:seedmissing}
\Pp_{\Simple}(T_a\ge n/2\mid E_H)
=\kappa(\lambda)\frac{H}{a_0}(1+o(1)),\qquad H=\lceil d^2\rceil.
\end{equation}
Entrance histories are not fresh uniform seed sets. The error
$1/\log d$ is larger than $H/a_0$, and $H/a_0$ does not satisfy the
shrinking-coordinate condition in~\eqref{eq:seedabsolute}.
At a diffusion horizon $Su_0$, the raw early TV bound is $O(Sb_0)$,
whose ratio to a fixed-seed target $ab_0$ is $O(S/a)$; it cannot alone
supply relative simplicity transfer. Although the density proof is
multiplicative on every simple transcript, it also requires
\[
\Pp_{\Conf}(T_a\ge n/2,\text{nonsimple exposed history})
=o(\Pp_{\Conf}(T_a\ge n/2))
\]
and relative cap/physical-tail control to prove the conjecture. These
survival-relative statements have not been proved here.

There is a precise obstruction to extending the crude no-repeat tree
coupling until the mesoscopic comparison becomes usable. If
$N=H^2\ell$ steps are used, its error is $O(d^2H^4\ell^2/n)$.
Relative entrance accuracy on the $a/H$ scale requires
$d^2H^5\ell^2/n\to0$, whereas~\eqref{eq:seedabsolute} needs
$H\gg a_0/\log d$. At logarithmic degrees these demand respectively
$H\ll n^{1/5}$ and $H\gg n^{1/3}$ up to logarithmic factors.
They have no overlap for this argument; this does not disprove the
fixed-seed conjecture or rule out refined methods with repeated vertices.

The late envelope~\eqref{eq:qXlate} is already $o(a_0^{-1})$ in the
logarithmic regime. Thus late macroscopic closure is not the missing
relative estimate. Also, for bounded positive $\lambda$,
\eqref{eq:BboundedCorrection} is $O(1/(a_0d))$, negligible relative
to $a/a_0$ for fixed $a$; exact centering is not an additional obstruction
there. The obstruction is small-state continuation and relative early transfer.

An elementary logical counterexample makes the limitation of C explicit.
Define increasing probability arrays $Q_n(a)=p(a/a_0,\lambda)$ except
$Q_n(1)=\tfrac12p(1/a_0,\lambda)$, and $Q_n(0)=0$.
Then $\sup_a|Q_n(a)-p(a/a_0,\lambda)|\to0$, but
$Q_n(1)/(\kappa(\lambda)/a_0)\to1/2$.
This is a counterexample to extrapolating an absolute theorem, not a
graph-model counterexample. The final fixed-$a$ law is not established here.

\section{Probability inequalities used in the article}\label{app:prob}
This appendix gives the precise elementary inequalities used above and their
proofs. Conditional expectation, integration, independence, and Gaussian
random variables have their usual measure-theoretic meanings.

\begin{lemma}[Doob's second-moment bound]\label{lem:Doob}
For a square-integrable right-continuous martingale $N$ on $[0,R]$,
$\E\sup_{t\le R}|N(t)|^2\le4\E|N(R)|^2$.
\end{lemma}
\begin{proof}
First take a finite time grid and let $N^*$ be its maximum absolute value.
Stop at the first grid point with $|N|\ge a$. Conditional Jensen and the
martingale property give
$a\Pp(N^*\ge a)\le\E[|N(R)|\ind_{\{N^*\ge a\}}]$.
Integrate against $2\dd a$ and use Tonelli to get
$\E(N^*)^2\le2\E[|N(R)|N^*]$.
Cauchy--Schwarz yields the bound. Increasing finite grids containing a dense
set of times converge to the path supremum by right continuity; monotone
convergence finishes the proof. The same argument applies to a martingale
stopped at a bounded stopping time, by approximation with grid stopping times.
\end{proof}

\begin{lemma}[Bernstein and compensated jump martingales]\label{lem:exp}
If independent centered variables $\xi_i$ satisfy $|\xi_i|\le J$ and
$\sum_i\E\xi_i^2\le V$, then
\[
\Pp(|\sum_i\xi_i|>e)\le2\exp[-e^2/(2(V+Je/3))].
\]
If a compensated jump martingale $N$ has jumps of absolute size at most $J$
and predictable bracket at most a deterministic $V$ through time $R$, then
\[
\Pp(\sup_{t\le R}|N(t)-N(0)|>e)
\le2\exp[-e^2/(2(V+Je/3))].
\]
The second statement also holds conditionally on the information at the
start of a deterministic interval, with that interval's bracket increase.
\end{lemma}
\begin{proof}
For $|x|\le J$ and $0<\theta<3/J$, the exponential series and
$k!\ge2\,3^{k-2}$ for $k\ge2$ imply
\[
e^{\theta x}-1-\theta x
\le\frac{\theta^2x^2}{2(1-\theta J/3)}.
\]
For an independent centered variable this gives
$\E e^{\theta\xi}\le\exp[\theta^2\E\xi^2/(2(1-\theta J/3))]$.
Multiply the moment generating functions and use Markov's inequality.
Choosing $\theta=e/(V+Je/3)$ gives the positive tail; repeat with $-\xi$
and add the two bounds.

For the jump martingales here, write
$N=\sum_j\int g_j(\dd N_j-\lambda_j\dd t)$, with predictable intensities
and predictable jumps $|g_j|\le J$. Conditioning over a short interval,
or applying the product rule between jumps, shows that
\[
\exp\!\left(\theta[N(t)-N(0)]-\sum_j\int_0^t
\lambda_j(s)[e^{\theta g_j(s)}-1-\theta g_j(s)]\dd s\right)
\]
is a nonnegative local martingale, hence has expectation at most one.
For detail, a type-$j$ jump has conditional probability
$\lambda_j\dd t+o(\dd t)$; its multiplier is $e^{\theta g_j}$, while the
continuous compensator multiplier cancels exactly that conditional first
order change. Multiplication over partitions and localization justify the
identity; the processes here have finitely many jumps and bounded intensities
on bounded horizons, so dominated limits apply after localization.
The series bound makes
$\exp(\theta[N-N(0)]-\theta^2\langle N\rangle/[2(1-\theta J/3)])$
a supermartingale. Stop at the first crossing of $e$ or at $R$.
The same choice of $\theta$ and the bracket bound give the positive maximal
tail. Apply to $-N$ and add. Starting at an interval endpoint gives the
conditional assertion by the identical conditional expectation calculation.
\end{proof}

\begin{lemma}[Hoeffding, Azuma, and Poisson tails]\label{lem:basicexp}
For independent $Y_i$ in intervals of lengths $D_i$,
\[
\Pp(|\sum_i(Y_i-\E Y_i)|>e)\le2\exp[-2e^2/\sum_i D_i^2].
\]
For $q$ martingale differences of absolute size at most $D$, the bound is
$2\exp[-e^2/(2qD^2)]$. For a Poisson variable of mean $\mu$, a cap
$m\ge2\mu$ obeys $\Pp(N\ge m)\le e^{-c m}$ for an absolute $c>0$.
\end{lemma}
\begin{proof}
For a bounded variable of interval length $D$, put
$f(\theta)=\log\E e^{\theta Y}$. Its second derivative is a variance under
the tilted law. Every law on an interval of length $D$ has variance at most
$D^2/4$: its squared distance from the midpoint is at most $D^2/4$, and
variance minimizes the mean squared distance over constants. Thus two
integrations give $f(\theta)-f(0)-\theta f'(0)\le\theta^2D^2/8$.
Multiply and optimize at $\theta=4e/\sum_iD_i^2$ to prove Hoeffding.
For a martingale difference of absolute size $D$, conditional range length
is $2D$; iterate the conditional bound and optimize to obtain Azuma.
Finally the Poisson mass series gives
$\E e^{\theta N}=\exp[\mu(e^\theta-1)]$. At $\theta=\log2$,
$\Pp(N\ge m)\le\exp[\mu-m\log2]\le
\exp[-(\log2-1/2)m]$.
\end{proof}

The exact sampling formula~\eqref{eq:seedvar} follows by seed indicators
$I_i$ with $\E I_i=a/n$ and $\E I_iI_j=a(a-1)/(n(n-1))$ for $i\ne j$:
expanding the variance of $\sum_i I_ig_i$ gives
$a(n-a)/(n-1)$ times empirical variance with divisor $n$.
For a degree class it gives the hypergeometric variance used in
\eqref{eq:seeduniform}. The sampled-sum martingale increment in
\eqref{eq:seedexp} follows by averaging the remaining $a-j$ draws over the
remaining population after draw $j$ and subtracting the preceding conditional
mean; the result is exactly $(n-a)(g_{I_j}-\bar g_{j-1})/(n-j)$.

\section{Normal arrays and compact functional convergence}\label{app:limits}
\begin{lemma}[Small independent summands]\label{lem:CLT}
Suppose an independent centered triangular array has
$\max_i|\xi_{ni}|\le\eta_n\to0$ and
$\sum_i\E\xi_{ni}^2\to v<\infty$. Its sum tends in distribution to a
centered normal variable of variance $v$. The vector assertion holds when
every fixed linear combination satisfies the bounds and covariances converge.
\end{lemma}
\begin{proof}
At fixed $t$, Taylor expansion gives
\[
\E e^{it\xi_{ni}}=1-t^2\E\xi_{ni}^2/2+r_{ni},\qquad
|r_{ni}|\le |t|^3\eta_n\E\xi_{ni}^2/6.
\]
Take logarithms and sum. The summed remainders tend to zero, as do quadratic
logarithm errors, bounded by a constant times
$\max_i\E\xi_{ni}^2\sum_i\E\xi_{ni}^2\le
\eta_n^2\sum_i\E\xi_{ni}^2$. Thus characteristic functions tend to
$e^{-vt^2/2}$. To obtain distribution convergence directly, bounded second
moments give tightness by Chebyshev. On a large compact interval, approximate
a bounded continuous test function by Gaussian convolution of a smooth
compact cutoff. Fourier inversion for that integrable smooth function and
dominated convergence of characteristic functions give its expected limit.
The tightness bound controls the discarded tail and uniform continuity
controls smoothing error; increase the cutoff and shrink the smoothing
variance. For vectors apply the same expansion to every $t\cdot\xi_{ni}$
and use finite-dimensional Gaussian smoothing and Fourier inversion.
All estimates are uniform if the maximum and variance errors are uniform
over a conditional parameter, so averaging preserves the limit.
\end{proof}

\begin{lemma}[Mesh oscillations and the continuous functional limit]\label{lem:compact}
Let càdlàg processes $X_n$, regarded as random elements of $D([0,R])$,
have finite-dimensional laws
converging to a continuous-path process $X$. If for every $\varepsilon>0$,
\[
\lim_{\delta\downarrow0}\limsup_n
\Pp\!\left(\sup_{|s-t|\le\delta}|X_n(s)-X_n(t)|>\varepsilon\right)=0,
\]
then $X_n$ converges in $D([0,R])$ with the Skorokhod topology to $X$.
In particular minima converge, as do bounded Borel-measurable functionals
that are continuous at every continuous path.
\end{lemma}
\begin{proof}
On a fixed mesh of size $\delta$, let $I_\delta X_n$ be linear interpolation
of the mesh values. Its uniform distance from $X_n$ is at most twice the
oscillation at distance $\delta$. On a fixed mesh interpolation is a
continuous finite-dimensional map, so its law tends to that of $I_\delta X$.
For a bounded Lipschitz functional in the Skorokhod metric compare expectations on
$X_n$ and $I_\delta X_n$ by the oscillation on the good event and by twice
the bound otherwise. The hypothesis removes the error as $\delta\downarrow0$;
continuity of $X$ gives $I_\delta X\to X$ almost surely uniformly. This
proves the assertion for the bounded Lipschitz determining class of the
path metric. The same compact interpolations and oscillation bounds give
continuous-limit convergence in the Skorokhod metric for the original
càdlàg paths. The minimum is continuous in the Skorokhod metric at
continuous paths and one-Lipschitz in the uniform metric,
and threshold probabilities converge wherever its limiting law has no atom.
\end{proof}
The exponential mesh-cell estimate in Section~\ref{sec:X} gives precisely
this oscillation bound. Its martingale correction tends to zero uniformly
in probability and its finite-dimensional laws follow from
Lemma~\ref{lem:CLT}. Every step of the functional convergence is therefore
included, without borrowing a martingale functional central limit theorem.

\section{Brownian boundary facts}\label{app:Brownian}
Brownian motion is the continuous centered Gaussian process with covariance
\[
\E\mathcal B(s)\mathcal B(t)=\min(s,t).
\]
This covariance implies independent
stationary Gaussian increments and
$\mathcal B(\lambda t)/\sqrt\lambda\overset{d}=\mathcal B(t)$.

\begin{lemma}[Reflection and deterministic drift]\label{lem:Brownianmin}
For $h>0$,
$\Pp(\min_{s\le R}\mathcal B(s)\le-h)=2\Pp(\mathcal B(R)\le-h)$.
On a finite interval, adding $g(t)=\int_0^t f(s)\dd s$ for continuous
deterministic $f$ preserves Brownian null events. Thus for $x>0$,
$\min_{s\le R}[x+\mathcal B(s)+g(s)]$ has no atom at zero; for $x=0$
the drifted path is negative somewhere in every interval $[0,R]$, $R>0$,
almost surely.
\end{lemma}
\begin{proof}
At the first hit of $-h$, reflect all future increments. Symmetry and
independence preserve their law and put paths ending above $-h$ in
one-to-one probability-preserving correspondence with those ending below
$-h$; the endpoint equality has zero Gaussian probability. Independence
at the hitting time can be justified without another theorem: approximate
it from above by mesh times $2^{-k}$. At each mesh stopping time, partition
by its value and use deterministic-time independent increments. Continuity
and convergence of the mesh times give the same independence and symmetry
at the hitting time, first for finite cylinder events, then for the
continuous path sigma field. This proves reflection.

Construct $I_f=\int_0^R f\dd\mathcal B$ as the mean-square limit of step
functions. For step $f$, Gaussian increments show the strictly positive
density $L_f=\exp(I_f-\tfrac12\int f^2)$ has expectation one. Under this
density a finite Brownian vector has its original covariance and mean
$g(t)$, by multiplying its Gaussian exponential moment by $L_f$.
Approximate continuous $f$ in $L^2$. Integrals converge in mean square,
drifts converge uniformly by Cauchy--Schwarz, and densities converge in
$L^1$: their second moments $e^{\|f\|_2^2}$ stay bounded and give uniform
integrability. The density identity follows for continuous $f$.
Cylinder sets at rational times generate the continuous path sigma field,
so the law of $\mathcal B+g$ has this density relative to Brownian law.
Strict positivity gives equivalence of null events.

For $x>0$ reflection makes the Brownian minimum law continuous at $-x$;
the drift density preserves the zero probability of that exact minimum.
For $x=0$, let $h\downarrow0$ in reflection to get probability one of a
negative minimum on each $[0,R]$, $R>0$. Taking countably many positive
rational $R$ and using absolute continuity gives the assertion for the drifted path.
\end{proof}

\begin{lemma}[Growth, tubes, and linear-drift survival]\label{lem:Brownianother}
Almost surely $\mathcal B(s)=o(s^2)$ as $s\to\infty$.
For every $R,\eta>0$, $\Pp(\sup_{s\le R}|\mathcal B(s)|<\eta)>0$;
hence a positive-width tube around any continuously differentiable path
starting at zero has positive probability. If $\mu,h>0$,
\[
\Pp(\inf_{s\ge0}[\mathcal B(s)+\mu s]\le-h)\le e^{-2\mu h}<1.
\]
\end{lemma}
\begin{proof}
Doob gives
$\Pp(\sup_{s\le2^k}|\mathcal B(s)|>\varepsilon2^{2k})
\le4\,2^{-3k}/\varepsilon^2$. The sum is finite; a union bound on each
tail shows only finitely many such events occur almost surely. Apply for
rational positive $\varepsilon$ and compare neighboring dyadic intervals
to obtain subquadratic growth.

For a small ball partition $[0,R]$ into sufficiently short intervals of
length $h_*$. Starting from $|y|\le\eta/4$, the probability that the next
endpoint is again in $[-\eta/4,\eta/4]$ is bounded below near $1/2$ as
$h_*\downarrow0$. The required Gaussian increment interval has length
$\eta/2$, center between $-\eta/4$ and $\eta/4$; its minimum mass occurs
at an extreme and tends to $1/2$. Reflection and a union of signs bound
the chance its increment ever exceeds $\eta/2$ in absolute value by
$4\Pp(N(0,h_*)\ge\eta/2)\to0$. Choose $h_*$ so the endpoint and excursion
conditions have joint conditional probability at least $1/4$. On that event
the segment stays within $3\eta/4$ of zero and ends within $\eta/4$.
Repeating over finitely many segments has probability at least $(1/4)^N>0$.
The deterministic Gaussian shift in Lemma~\ref{lem:Brownianmin} transfers
this positive probability to a tube around a smooth path.

Finally $\exp[-2\mu(\mathcal B(s)+\mu s)]$ is a nonnegative martingale,
as the Gaussian increment moment cancels its drift. Stop at the first
hit of $-h$ or deterministic $R$; on the hit its value is $e^{2\mu h}$,
while stopped expectation is at most one. Thus hit probability by $R$
is at most $e^{-2\mu h}$; let $R\to\infty$. Deterministic-time independent
future increments give the conditional continuation bound used above.
\end{proof}

\section{Integer arithmetic and graphical realizability}\label{sec:arith}
Writing $V=\Var(D)$, direct substitution into~\eqref{eq:bK} gives
\begin{equation}\label{eq:criticalV}
b=\frac{(d-c)(c+1-v)+c^2}{d^2}
=\frac{d-c}{d^3}(V_*-V),\qquad
V_*=d+\frac{cd^2}{d-c}.
\end{equation}
Thus $v=c+1$ gives $b=c^2/d^2$, not zero. If $C^2<c+1$, the envelope gives
$b\ge[(d-c)(c+1-C^2)+c^2]/d^2$, of order $d^{-1}$. If $C^2=c+1$, it gives
$b\ge c^2/d^2>0$. With integer mean there is the additional bound
$V\le\lfloor C\sqrt d\rfloor^2$, so endpoint radius rounding may make the
lower bound larger. A strict capacity margin $C^2>c+1$ is needed to cross zero.

For integer $d$, write $x_i=d_i-d$. Since $\sum_i x_i=0$ and
$x_i^2\equiv x_i\pmod2$, necessarily
\begin{equation}\label{eq:lattice}
V=2Q/n\ (Q\in\mathbb Z),\qquad
b_Q=\frac{d-c}{d^3}(V_*-2Q/n),\qquad
\Delta_b=\frac{2(d-c)}{nd^3}.
\end{equation}
At a fixed rational nonintegral mean $d=J/n$, the identity
$\sum_i d_i^2\equiv J\pmod2$ likewise puts possible variances on an affine
lattice with spacing $2/n$. Full attainment below is claimed only for integer
means. Even $n$ is a convenient sufficient parity choice, not a necessary one.

\begin{lemma}[A graphicality switch]\label{lem:graphical}
Let $n$ be even and let $d,r$ be nonnegative integers. Put
$\delta=d-r\ge2$, $\Delta=d+r<n$. If
\begin{equation}\label{eq:graphcond}
(n-2\Delta-3)\delta>(2\Delta+3)\Delta,
\end{equation}
every integer sequence in $[\delta,\Delta]$ with sum $nd$ is graphical.
\end{lemma}
\begin{proof}
A simple $d$-regular graph exists on $\mathbb Z/n\mathbb Z$: use offsets
$\pm1,\ldots,\pm d/2$ for even $d$, or offsets
$\pm1,\ldots,\pm(d-1)/2$ together with the antipodal matching for odd $d$.
All edges are distinct since $d<n$. Among simple graphs with $nd/2$ edges
and all degrees in $[\delta,\Delta]$, choose one minimizing
$\sum_i(\deg(i)-d_i)^2$. If the discrepancy is positive, equal total degrees
give a deficit vertex $u$ and a surplus vertex $v$. If an $x\in N(v)\setminus
\{u\}$ is not adjacent to $u$, replace $vx$ by $ux$. The degree bounds are
preserved and the discrepancy drops by at least two.

Otherwise choose $x\in N(v)\setminus\{u\}$, possible since $\deg(v)\ge2$.
Let $B_*=V(G)\setminus(N(u)\cup N(x)\cup\{u,v,x\})$. Its size is at least
$n-2\Delta-3$. If it were independent, edge counting would give
$|B_*|\delta\le(n-|B_*|)\Delta\le(2\Delta+3)\Delta$, contrary
to~\eqref{eq:graphcond}. Choose an edge $yz$ inside $B_*$. Replace $vx,yz$
by $xy,uz$. The five vertices are distinct and the new edges were absent.
Only $u$ gains and $v$ loses a degree, preserving the bounds and again reducing
the discrepancy. Minimality is contradicted, so the discrepancy is zero.
\end{proof}

\begin{proposition}[Full variance lattice near criticality]\label{prop:realize}
Fix $0<a_-<c+1<a_+<L<C^2$. For all sufficiently large integer $d$ and all
even $n\ge K_0d$, every $Q\in\mathbb Z$ within one of $nV_0/2$ for some
$V_0\in[a_-d,a_+d]$ is attained by a graphical, symmetric empirical law on
at most the five values $d,d\pm s,d\pm(s+1)$, where $s=\lceil\sqrt{Ld}\rceil$.
One sufficient choice is
\[
K_0=\max\left\{8,\frac{2(H_L+1)}{a_-},\frac2{L-a_+}\right\},\qquad
H_L=(\sqrt L+1)^2(\sqrt L+2)^2.
\]
The actual variance is $2Q/n$ and the envelope~\eqref{eq:envelope} holds.
\end{proposition}
\begin{proof}
Take $d\ge\max\{8,16C^2,4/(C-\sqrt L)^2\}$. Put
$F_s=(s^2-1)(s+1)^2$. Since $F_s\le H_Ld^2$, $n\ge K_0d$ implies
\[
n\ge8d,\qquad na_-d/2\ge F_s+1,\qquad n(L-a_+)d/2\ge1.
\]
The coprime integers $s^2,(s+1)^2$ give a unique
$\ell\in\{0,\ldots,s^2-1\}$ such that $(s+1)^2\ell\equiv Q\pmod{s^2}$.
Put $m=[Q-(s+1)^2\ell]/s^2$. Because $Q\ge F_s$, $m$ is a nonnegative
integer. Also
\[
m+\ell=Q/s^2-(2s+1)\ell/s^2
\le(na_+d/2+1)/s^2\le nLd/(2s^2)\le n/2.
\]
Use $m$ copies of each of $d\pm s$, $\ell$ copies of each of $d\pm(s+1)$,
and $n-2(m+\ell)$ copies of $d$. The counts are nonnegative integers,
opposite deviations give exact mean $d$, and the square sum is $2Q$.
Furthermore $s+1\le\sqrt{Ld}+2\le C\sqrt d\le d/4$.
For $r=s+1$, $\delta\ge3d/4$, $\Delta\le5d/4$, and
\[
n\delta\ge6d^2>5d^2+6d\ge(2\Delta+3)(\delta+\Delta),
\]
which is~\eqref{eq:graphcond} rearranged. Lemma~\ref{lem:graphical} applies.
\end{proof}

For a target $t$ whose variance
$V_t=V_*-d^3t/(d-c)$ lies in the proposition's interval, take $Q$ nearest
to $nV_t/2$. Then
\begin{equation}\label{eq:tuning}
|b-t|\le\frac{d-c}{nd^3}=\Delta_b/2.
\end{equation}
This is the optimal nearest-lattice error at these fixed $n,d$: every degree
sequence lies on~\eqref{eq:lattice}, and a nearest point is attained. It is
not a positive lower bound for every target. In particular
\[
t=\beta/d^2\quad\Longrightarrow\quad
V_t=(c+1)d+\frac{d(c^2-\beta)}{d-c},\qquad
b=\beta/d^2+O(1/(nd^2)),
\]
for each fixed real $\beta$. Inclusion in the interval is eventual; the
radius cutoff alone is not claimed to guarantee it. Neighboring lattice
points below and above $nV_*/2$ give strictly positive and negative $b$
of magnitude at most $\Delta_b$, using the adjacent two points if the
critical coordinate is integral. For a nonzero varying target, relative
tuning follows when $nd^2|t|\to\infty$, together with interval inclusion.

\begin{proposition}[Exact zero and its arithmetic obstruction]\label{prop:zero}
If $c=p/q>0$ is rational in lowest terms, integer $d>p/q$, and $n=2N$, put
$h_d=(qd-p)/\gcd(qd-p,pd^2)$. Exact $b=0$ requires $h_d\mid N$ and,
in the eventual range of Proposition~\ref{prop:realize}, this condition is
also sufficient. If $c$ is irrational, no finite empirical degree law with
any mean can have exact $b=0$.
\end{proposition}
\begin{proof}
For rational $c$, $V_*=d+pd^2/(qd-p)$ has reduced denominator $h_d$.
Equation~\eqref{eq:lattice} gives $b=0$ exactly when $NV_*\in\mathbb Z$,
equivalently $h_d\mid N$. The construction attains that point when $V_*$
is in its interior interval. Since $h_d=O_c(d)$, there are allowed sizes
$n\ge K_0d$ with $n=O_c(d)$ and this divisibility, and arbitrarily larger
ones. For an arbitrary finite empirical law, $d$ and $H=\E[D(D-1)]$ are
rational. If $b=0$, $H>0$ and $c=d-d^3/H$, which is rational.
\end{proof}
For example $c=1$ gives $h_d=d-1$. For irrational $c$ the exact smallest
absolute deficit in the constructed range is
$\Delta_b\operatorname{dist}(nV_*/2,\mathbb Z)>0$, with no uniform
Diophantine lower bound claimed. The construction also has regular moments:
when $V=(c+1)d+O(1)$,
\[
(m+\ell)/n=(c+1)/(2L)+O(d^{-1/2}),\qquad
\mu_3=0,\qquad \mu_4/d^2\to L(c+1).
\]
Indeed the first assertion follows from the formula for $m+\ell$ and
$\ell<s^2=O(d)$; for the last use
$\mu_4=s^2V+(2\ell/n)(s+1)^2(2s+1)$, whose second term is $O(d^{3/2})$.

\section{Higher-order barrier corrections and moment comparisons}\label{app:moments}
This appendix completes the expansions in Theorem~\ref{thm:D} and
compares empirical laws sharing their first two or three moments.
We use $u=b/K$, $A=Ku^2/2$, $N(z)=z-F(z)$ and
$g(z)=B(z)-(1-z)F'(z)$.

\subsection{Uniform higher-order expansion}
\begin{proof}[Derivation of~\eqref{eq:location}--\eqref{eq:Drefined}]
For the refinement, the exact derivative formula for $2\le j\le5$ is
\[
F_k^{(j)}(z)=(-1)^j(k-1)_j(1-z)^{k-j-2}
[(j-1-j\pi)(1-z)-(1-\pi)(k-j-1)z].
\]
It implies $|F^{(j)}|=O(d^j)$ throughout the maximizing interval. Hence
\begin{align*}
N(z)&=bz-Kz^2/2-T_3z^3/6-U_4z^4/24+O(d^5z^5),\\
g(z)&=b-Kz+(K-T_3)z^2/2+O(d^4z^3).
\end{align*}
The finite bounds give $\widehat z-u=O(du^2)$ first. Substitution in the
second display and division by $K$ give~\eqref{eq:location}, since
$du=O(d^{-2})$. Put $q=(K-T_3)/(2K)$. Substitution into the first display gives
\[
N(\widehat z)=A-T_3u^3/6+
[T_3^2/(8K)-K/8-U_4/24]u^4+O(d^5u^5).
\]
The location remainder is harmless: $N'(u+qu^2)=-Ku^2/2+O(d^4u^3)$,
whose product with $O(d^2u^3)$ is absorbed in the stated error, as is its
second-order displacement error. Now
$F(\widehat z)=u+(q-K/2)u^2+O(d^3u^3)$. Expanding
$(1-F)^{-1}=1+F+F^2+O(u^3)$ and multiplying gives precisely
\eqref{eq:Drefined}. The omitted terms are $O(d^5u^5)=O(b^5/d^5)$.
\end{proof}

\subsection{Third- and fourth-moment comparisons}
For two laws sharing $d,v$, their exact $b$ agrees and
$\Delta K=(1-2\pi)\Delta\mu_3/d$. Since
$\Delta K=O(d^{1/2})$, $\Delta T_3=O(d^{3/2})$, and
$\Delta U_4=O(d^{5/2})$, subtracting~\eqref{eq:Drefined} yields
\begin{equation}\label{eq:skew}
\adet^{(2)}-\adet^{(1)}
=-\frac{(1-2\pi)b^2\Delta\mu_3}{2dK_1K_2}
+O(b^3/d^{9/2}).
\end{equation}
To check the orders directly, the varying parts of $\E(D)_4/d$ and
$\E(D)_5/d$ are respectively
$(4-6/d)\mu_3+\mu_4/d$ and
$(10d-40+35/d)\mu_3+(5-10/d)\mu_4+\mu_5/d$.
If also $\mu_{3,1}=\mu_{3,2}$, then $K$ agrees and
\[
\Delta T_3=-(2-3\pi)\Delta\mu_4/d=O(d),\quad
\Delta U_4=(3-4\pi)[(5-10/d)\Delta\mu_4+\Delta\mu_5/d]=O(d^2).
\]
The $T_3^2/K$ and $U_4$ differences times $u^4$ are $O(b^4/d^6)$;
the other fourth-order differences are smaller. The separate fifth-order
remainders are absorbed because $bd\le c+1$. Thus
\begin{equation}\label{eq:kurt}
\adet^{(2)}-\adet^{(1)}
=\frac{(2-3\pi)b^3\Delta\mu_4}{6dK^3}+O(b^4/d^6).
\end{equation}
These comparisons are deterministic. If a displayed leading moment
difference is exceptionally small, no dominance over its remainder is claimed.

\section*{Author disclosure}
The author is employed by Yandex. This research was conducted independently
in the author's personal time and received no dedicated funding from Yandex.
Some computational work used Yandex equipment and computing infrastructure.
The views expressed are solely those of the author and do not necessarily
represent those of Yandex. The author declares no other competing interests.

\section*{AI assistance}
Generative AI assistants, including OpenAI GPT-6.1 and GPT-6 Astra models,
were used for literature analysis, mathematical calculations and proof
development, manuscript drafting and editing, and checks of arguments
and references. OpenResearch was used to organize research tasks and
supporting artifacts. The author is responsible for all mathematical
results and the final manuscript.

\end{document}